\documentclass[12pt]{iopart}

\usepackage{amsmath}
\usepackage{amssymb}
\usepackage{amsfonts}
\usepackage{mathrsfs}
\usepackage{amsthm}
\usepackage{epsfig}
\usepackage{graphicx}
\usepackage{multicol}
\usepackage{crop}
\usepackage{float}
\usepackage{booktabs}
\usepackage{algorithmic}
\usepackage{algorithm}
\usepackage{hyperref}
\usepackage{xcolor}
\usepackage{tikz}
\usepackage[colorinlistoftodos]{todonotes}
\usepackage{soul}
\usetikzlibrary{quotes,angles}

\newtheorem{remark}{Remark}[section]
\newcommand{\mR}{\mathbb{R}}
\newcommand{\mN}{\mathbb{N}}
\newcommand{\mZ}{\mathbb{Z}}

\newcommand{\norm}[1]{\|#1\|}

\newcommand{\mbf}[1]{\mathbf{#1}}

\newcommand{\transp}[1]{{#1}^{\textup T}}

\newcommand{\argmin}{\operatornamewithlimits{arg min}}

\newcommand{\bq}{\begin{eqnarray}}
\newcommand{\eq}{\end{eqnarray}}
\renewcommand{\Pi}{\mbox{\LARGE$\pi$}}

\definecolor{mybox}{HTML}{F7E9A8}

\newtheorem{theorem}{Theorem}[section]
\begin{document}

\title{Automatic alignment for three-dimensional tomographic reconstruction}

\author{
	Tristan van Leeuwen$^1$,
	Simon Maretzke$^3$
	and
	K. Joost Batenburg$^{2,4}$
}

\address{
	$^1$Mathematical Institute, Utrecht University, the Netherlands\\
	$^2$Centrum Wiskunde \& Informatica, Amsterdam, the Netherlands\\
	$^3$Institute for Numerical and Applied Mathematics, University of G{\"ottingen}, G{\"ottingen}, Germany\\
	$^4$Mathematical Institute Leiden University, the Netherlands
}

\begin{abstract}
In tomographic reconstruction, the goal is to reconstruct an unknown object from a collection of line integrals. Given a complete sampling of such line integrals for various angles and directions, explicit inverse formulas exist to reconstruct the object. Given noisy and incomplete measurements, the inverse problem is typically solved through a regularized least-squares approach. A challenge for both approaches is that in practice the exact directions and offsets of the x-rays are only known approximately due to, e.g.\ calibration errors. Such errors lead to artifacts in the reconstructed image. In the case of sufficient sampling and geometrically simple misalignment, the measurements can be corrected by exploiting so-called consistency conditions. In other cases, such conditions may not apply and we have to solve an additional inverse problem to retrieve the angles and shifts. In this paper we propose a general algorithmic framework for retrieving these parameters in conjunction with an algebraic reconstruction technique. The proposed approach is illustrated by numerical examples for both simulated data and an electron tomography dataset.
\vspace{1em}

\noindent \textbf{Published as:} \href{https://doi.org/10.1088/1361-6420/aaa0f8}{\emph{Inverse Problems} \textbf{34}(2):024004, 2018}\vspace{-2em}
% , including the case of truncated projections,
\end{abstract}

% 
%\maketitle

\section{Introduction}
In tomographic reconstruction, the aim is to recover an unknown object described by a function $u:\mathbb{R}^3\rightarrow \mathbb{R}_{+}$ with compact support from a finite number of samples of its X-ray transform
\[
p(\mathbf{s},\boldsymbol{\eta}) = \int_{\mathbb{R}}\!\mathrm{d}t\, u(\mathbf{s} + t\boldsymbol{\eta}),
\]
where $\mbf{s}\in\mathbb{R}^3$ and $\boldsymbol{\eta}\in\mathbb{S}^{2}$ determine the origin and direction of the X-ray (with $\mathbb{S}^{2} = \{\mbf{x}\in\mathbb{R}^3\,\, |\,\, \|\mbf{x}\|_2 = 1\}$ denoting the three-dimensional unit sphere). The samples of the X-ray transform are typically obtained by a scanning device, such as a CT scanner, where a source and detector traverse a certain trajectory around the object, which defines the set of sample points $\{(\mathbf{s}_i,\boldsymbol{\eta}_i)\}_{i=1}^m$. A few common setups are illustrated in figure \ref{fig:acquisition}. Here, a number of individiual ray-trajectories are illustrated as well as the axis and range of rotation of the object.

\begin{figure}[h!]
\centering
\begin{tabular}{ccc}
\includegraphics[scale=.25]{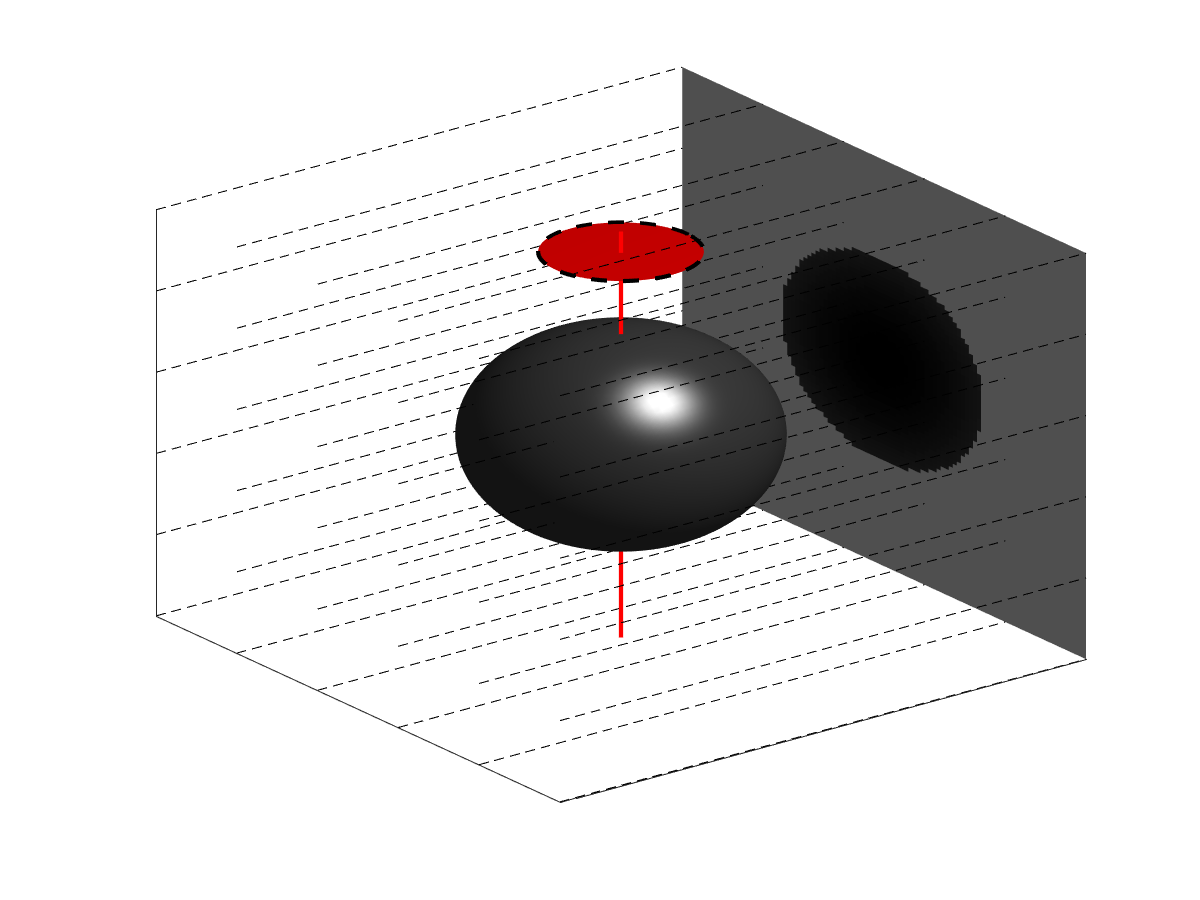}&
\includegraphics[scale=.25]{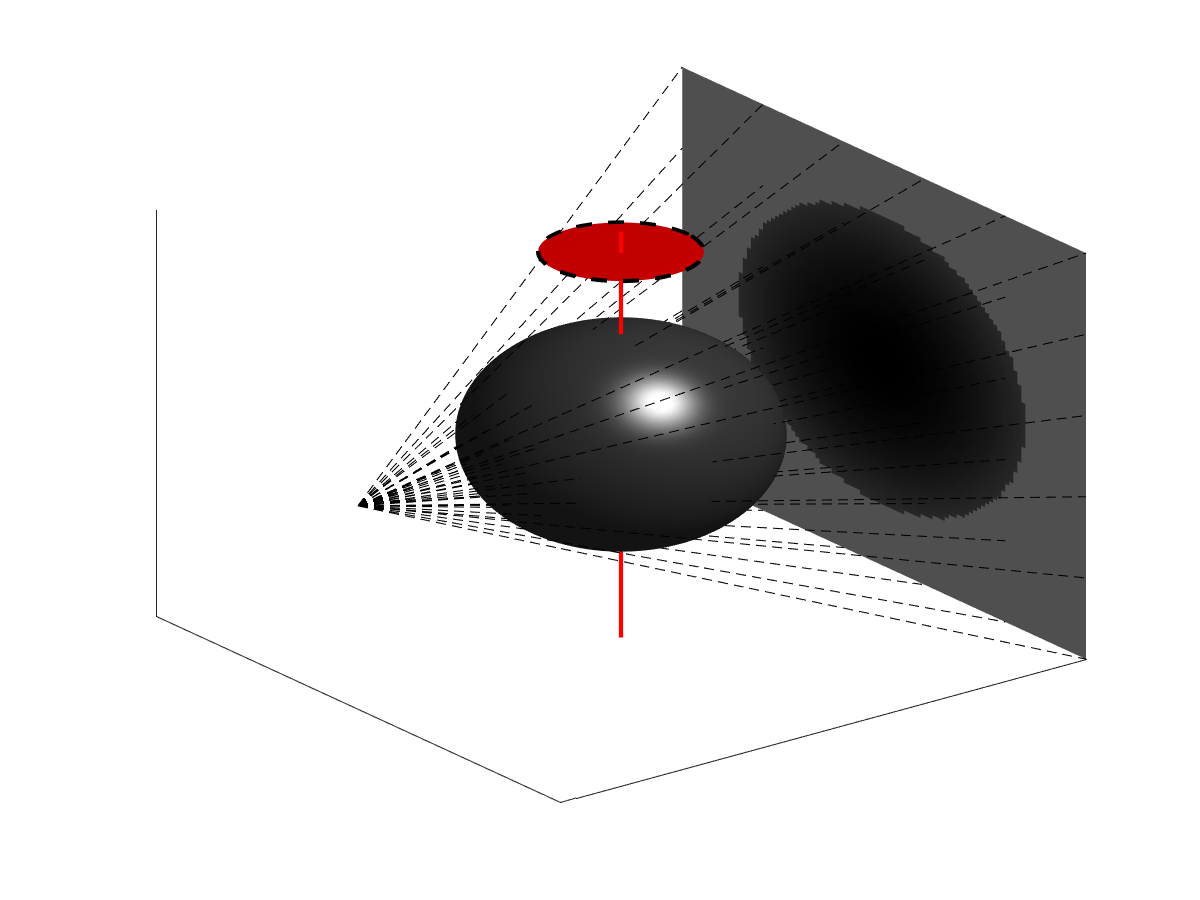}&
\includegraphics[scale=.25]{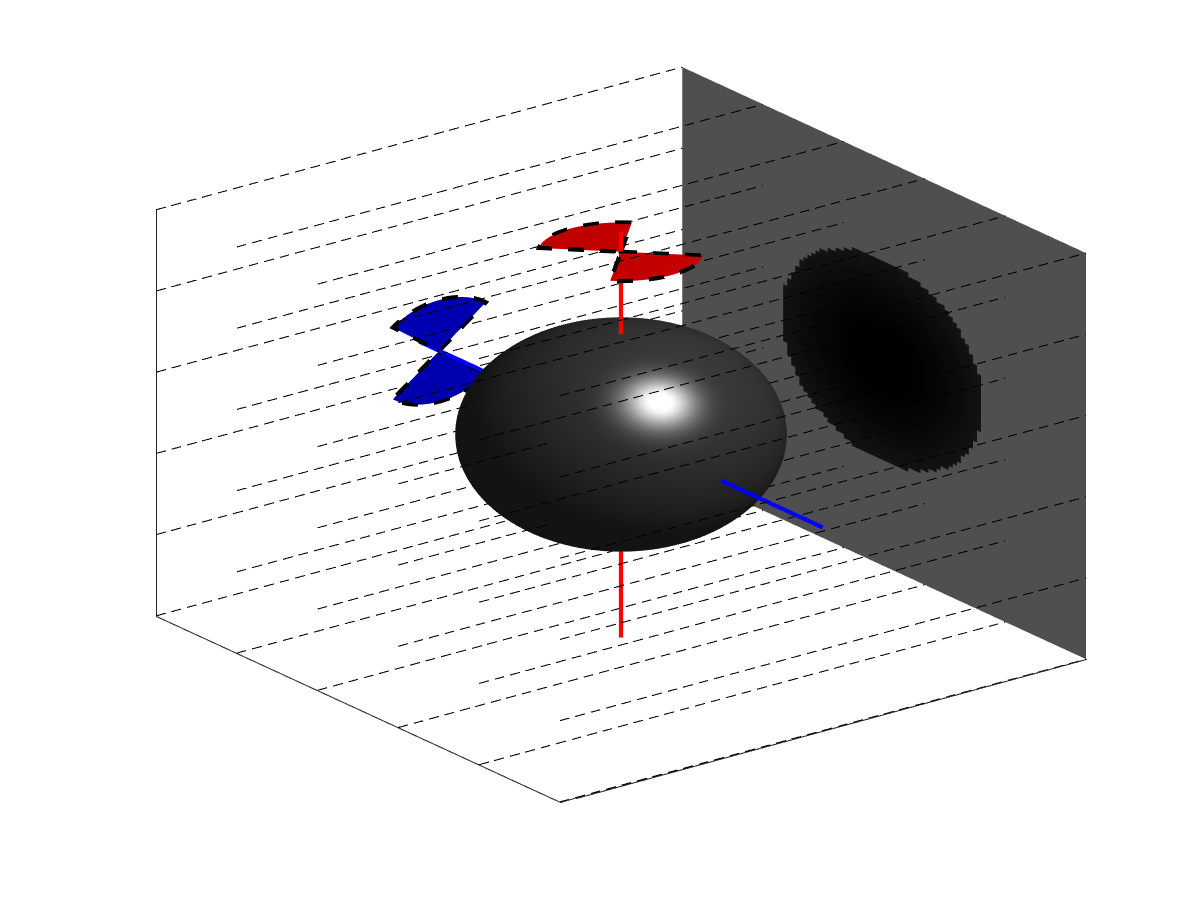}\\
{\small (parallel-beam)}&{\small (circular cone beam)}&{\small (dual tilt parallel-beam)}\\
\end{tabular}
\caption{Various acquisition setups that occur frequently in practice. Individual ray-trajectories are illustrated with a dotted line. The rotation axes of the object are denoted in red and blue while the angular range of the rotation is indicated with a (partial) disk. In the dual tilt setup, the object is rotated along the two axes subsequently.}
\label{fig:acquisition}
\end{figure}

The inverse problem consists of reconstructing $u$ from samples $\{p(\mathbf{s}_i,\boldsymbol{\eta}_i)\}_{i=1}^m$ of $p$. Upon Fourier transforming $p$ with respect to its first argument we get
\[
\widehat{p}(\boldsymbol{\xi},\boldsymbol{\eta}) = (2\pi)^{-\frac{3}{2}}\int_{\mathbb{R}}\!\mathrm{d}t\, \widehat{u}(\boldsymbol{\xi})\exp\left(-\imath t (\boldsymbol{\xi}\cdot\boldsymbol{\eta})\right),
\]
which forms the basis for many analytic reconstruction formulas. Indeed, for the parallel-beam geometry we have $\boldsymbol{\xi}\cdot\boldsymbol{\eta} = 0$ and we retrieve the well-known Fourier Slice Theorem, which states that the Fourier transform of $p$ along $\mathbf{s}$ gives the Fourier transform of $u$ along a slice perpendicular to $\boldsymbol{\eta}$ \cite{kak2001principles}. For some specific acquisition geometries, such as parallel-beam and fan-beam, explicit inverse formulas have been derived \cite{natterer_mathematics_2001}. These typically involve a filtering of the projection data followed by a backprojection step (filtered backprojection) or interpolation followed by an inverse Fourier transform. We can think of these as \emph{invert-then-discretize} approaches to solving the inverse problem. Being derived from a continuous model, these formulas only give an accurate reconstruction for fully and regularly sampled data.

On the contrary, various non-regular and under-sampled scenarios occur regularly in practice and can often be conveniently approached by first discretizing and then inverting the resulting system of linear equations. Such a \emph{discretize-then-invert} approach is often referred to as \emph{algebraic reconstruction} in the literature \cite{GORDON1970471,Herman09}.

Upon expressing $u$ in terms of an appropriate basis with coefficients $\{u_i\}_{i=1}^n$ and approximating the integral, the reconstruction problem can be posed as system of linear equations
\[
W\mathbf{u} = \mathbf{p},
\]
where $W \in \mathbb{R}^{m\times n}$ is the projection matrix whose elements depend on the acquisition parameters, $\mathbf{u}\in\mathbb{R}^n$ represents the object and $\mathbf{p}\in\mathbb{R}^m$ are the projection data for given rays characterized by $(\mbf{s}_i, \boldsymbol{\eta}_i)$, i.e., $p_i = p(\mbf{s}_i, \boldsymbol{\eta}_i)$. This moves the reconstruction problem into completely different realms of mathematics.
%For particular discretisations, uniqueness can be still be proven provided that the angular coverage is sufficient \ref{}. Under some strong assumptions on the discretized object $\mathbf{u}$, %uniqueness results can still be obtained for limited data. A notable example is \emph{discrete tomography}, where $u_i \in \{0,1\}$ and a unique reconstruction can be guaranteed for strongly %under-sampled data \ref{}. The reconstruction problem itself, however, becomes NP-hard.
Results and algorithms from the field of convex analysis and image processing can be brought to bear to include various regularization terms such as Total Variation to define a unique regularized solution. Although exact reconstruction can no longer be guaranteed with such approaches, they are often used in practice and may yield accurate results even in the case of highly corrupted or incomplete data \cite{Sidky09,Goris12}.

When dealing with tomographic problems in practice, a secondary type of inverse problem must often be solved as part of the image reconstruction process. In the standard models for
tomographic reconstruction, the geometry of the imaging system (position and 3D orientation of sources and detectors) are assumed to be perfectly known and are a fixed part of the
inverse problem that requires no further attention. In practice, however, the corresponding acquisition parameters, which determine the trajectories of the X-rays through the unknown object, may be known a-priori only with limited accuracy. Recovering these parameters then becomes an inverse problem of itself. An extreme case is single-particle analysis based on images obtained by cryo-electron microscopy, where the orientation of the object with respect to the incident X-rays is completely unknown \cite{Frank09}. Moreover, even if the acquisition parameters are prescribed by the setup, calibration errors or (rigid) movement of the object itself may result in a mismatch between the assumed ideal geometry and the effective geometry used to record the projections. Reconstruction from such distorted datasets introduces artifacts in the recovered object and leads to a loss of resolution. Estimating the acquisition parameters to correct such alignment errors is therefore of crucial importance in many practical applications \cite{Jing91,ParkinsonEtAl2012_AltMinAlignSoftXray,Mayo07}.

A relatively straightforward and practical approach to deal with the estimation of the alignment parameters is the introduction of markers attached to the scanned object, each of which can be tracked through the series of projection images. The motion of these markers (often small beads made of dense material) can then be related to their 3D position, determined from the series of observed 2D positions in the projection images. The strength of this approach lies in its ability to accurately recover many different geometrical parameters, including not only shifts of the projection images, but also detector tilts and variations in the position of the X-ray source. A significant drawback  is the need to prepare the object before the scan to apply the markers, and also the effects that the markers themselves can have on the reconstructed images: as marker particles have to be dense structures to show up in the individual projection images, they typically lead to streak artefacts in the reconstructions \cite{Song12}.

A common approach to the alignment problem that does not require the presence of markers is \emph{cross-correlation}.
%, which is based on the assumption that consecutive projections are similar.
%If the scanned object is very flat and is observed from the front,
The idea is that two projections with a small angular separation will be highly similar in structure, except for a geometric stretching of the data.
By incorporating such geometric assumptions in the alignment procedure and correlating consecutive images to each other sequentially, some shift and rotation parameters can be estimated in a heuristic way, but with limited accuracy. Out-of-plane movements on the other hand, such as detector tilts and variations in the position of the X-ray source, cannot be retrieved with such approaches.

%While these algebraic approaches allow for great flexibility in accounting for various tomographic setups, the corresponding acquisition parameters, which determine the paths of the X-rays through %the unknown object, may not be known a-priori. Recovering these then becomes an inverse problem of itself. An extreme case is single-particle analysis based on images obtained by cryo-electron %microscopy, where the orientation of the object with respect to the incident X-rays is completely unknown \ref{}. Even if the acquisition parameters are prescribed by the setup, calibration errors or %(rigid) movement of the object itself may result in a mismatch between the assumed, ideal geometry and the effective geometry used to record the projections. Reconstruction from such distorted %datasets introduces artifacts in the reconstruction and leads to a loss of resolution. Estimating the acquisition parameters to correct such alignment errors is therefore of crucial importance in %many practical applications.

% TODO Update and change this.

The category of alignment methods that do not require the use of markers or feature-detection in the projections is mainly based on exploiting the redundancy in the
tomographic data. Similar to the development of analytical and algebraic methods for solving the reconstruction problem, the alignment problem can also be approached along these two different paths. On the analytical
side, consistency conditions can be formulated for various source trajectories, including the well-known \emph{Helgason-Ludwig consistency conditions} \cite{Helgason99} for parallel-beam illumination. These prescribe sets of equations that must be satisfied by any valid set of tomographic projections. For example, the first order Helgason-Ludwig condition prescribes that the center of mass of the scanned object must project onto the center of mass of each of the measured projections, which can be used to estimate unwanted shifts of the detector that vary across the range of angles. A drawback of such methods is that they require full-field, i.e.\ non-truncated projection data. Hence, they are - strictly speaking - not applicable to  \emph{region-of-interest} tomography, although necessary conditions \cite{clackdoyle2015data} and heuristic adaptations \cite{sanders_physically_2015} have been proposed recently. Local consistency conditions can also be defined in some cases, e.g.\ based on John's equation for cone-beam illumination \cite{tang2012mathematical}. Although these conditions can be evaluated in a region-of-interest tomography setting, they instead rely on the assumption that the angular sampling is sufficient to allow for the numerical computation of derivatives that approximate a continuum of projection angles.  Recently, a new type of consistency conditions has been introduced in \cite{Aichert15} based on epipolar geometry that is more general and may lift some of these constraints, but its full relation to the classical consistency conditions remains to be investigated.

Alternatively, estimation algorithms for the alignment parameters have been developed based on minimizing the inconsistencies between the forward model applied to the reconstructed image and the measured projection data.
In  \cite{Yang05}, a quasi-newton algorithm is used to simultaneously compute the reconstructed image and the corresponding alignment parameters, using a limited-memory Broyden method to approximate the Hessian term. A similar approach is followed in
\cite{22162}, where a Levenberg-Marquardt method is used instead. In \cite{ParkinsonEtAl2012_AltMinAlignSoftXray}, a scheme for alignment of X-ray images is proposed that alternates between image reconstruction steps (using fixed alignment parameters) and alignment steps (using a fixed reconstruction). In addition to the consistency with the measured data, \cite{Houben2011_AltMinAlign} employs a sharpness criterion in the image domain as part of the optimization process.
This broad class of methods, known as \emph{projection matching}, follow a \emph{discretize-then-invert} approach and are much more general than the analytical approach. Similar to the algebraic reconstruction methods, they allow for the incorporation of prior knowledge about the object within the cost function for the parameter estimation and they can deal effectively with non-standard geometrical setups in principle, although existing implementations are often limited to correcting in-plane misalignment.

In this paper, we consider a rather general formulation of the projection matching approach and investigate various mathematical and numerical aspects of different solution strategies. Although a broad variety of projection matching algorithms have been described in the literature, they are typically proposed in relation to a specific application, where a single algorithm variant is proposed evaluated based on empirical results only. Instead, we focus on the underlying numerical mechanisms and outline the different choices that can be made in splitting the problem into subproblems.

We consider the setting where alignment errors can be modelled by three-dimensional rigid motion of the object.
% In particular, we assume that the projections are naturally grouped in \emph{projection images}, that represent the projections of the object from a particular direction.
Accordingly, each projection has its own set of up to six alignment parameters (since any rigid motion in three dimensions can be described by three shifts and three rotations). We assume that for each projection, the corresponding alignment parameters are approximately known. This is true for most tomographic imaging setups, where projection images are acquired from known directions and only small deviations from the known geometry need to be determined, but not in the field of single-particle-analysis for example, where projections can be randomly distributed with completely unknown angles. In such cases, different algorithms have to be applied such as the approach in \cite{Coifman08} based on graph laplacians.
Note that the alignment parameters for each projection image are considered as independent, and thus can be used to model model rigid motion of the object relative to the fixed source-detector-system during the data acquistion.

% Simon: the following formulation does not make much sense to me: why does it ``also account for certain source-detector misalignments'' BECAUSE ``the alignment parameters for each projection image are considered as independent''?
% Note that the alignment parameters for each projection image are considered as independent, and thus also account for certain source-detector misalignments as well as rigid motion of the object.

%(or a combination of all three effects).
The forward model for the tomographic alignment problem can be expressed as
\begin{equation}
\label{eq:jointFwModel}
W(\mathbf{a})\mathbf{u} = \mathbf{p},
\end{equation}
where $\mathbf{a}\in\mathbb{R}^\ell$ with $\ell = 6 N_{\textup{proj}}$ is a vector containing a suitable parametrization of a rigid object motion, $N_{\textup{proj}} \in \mN$ denoting the number of projection images. We propose an algorithmic approach for joint alignment and tomographic reconstruction based on the formulation \eqref{eq:jointFwModel}. This paper provides a mathematical framework that encompasses a large class of such methods, allowing for analysis in a common framework.

The paper is organized as follows. In section \ref{NLLS}, we present the main ideas underlying our approach and provide a detailed analysis of the proposed algorithm. Some practical aspects of the implementation of the algorithm are discussed in section \ref{algorithm}. Results on both phantom and real datasets are presented in section \ref{results}. Section \ref{conclusions} concludes the paper.

\section{Non-linear least-squares}
\label{NLLS}
We pose the joint reconstruction and alignment problem as a non-linear least-squares problem
\begin{equation}
\label{eq:joint}
\min_{\mathbf{a},\mathbf{u}} f(\mbf{a},\mbf{u}) = \textstyle{\frac{1}{2}}\|W(\mathbf{a})\mathbf{u} - \mathbf{p}\|_2^2 + h(\mathbf{u}),
\end{equation}
where $h$ is an appropriate regularization term. This problem is related to the Total Least-Squares (TLS) problem in the sense that the matrix itself is (to some extent) unknown \cite{TLS}. In the traditional TLS formulation, however, the perturbation $W$ would be modelled as a matrix-valued unknown, i.e., $W + \Delta W$. Although this provides an interesting viewpoint on the alignment problem, it is outside the scope of this paper to pursue this connection further.

Throughout the paper we will make the following assumptions on $f$:
\begin{itemize}
\item[(A1)] $W(\mathbf{a})$ has full rank for all admissable $\mathbf{a}$ or $h$ is \emph{strongly} convex. It follows from these assumptions that $f$ is strongly convex in $\mathbf{u}$ for any admissible $\mathbf{a}$ so that the solution $\overline{\mbf{u}}(\mathbf{a})$ that minimizes $f(\mathbf{a},\mathbf{u})$ is unique.
\item[(A2)] There exists an optimal set of alignment parameters and image, $(\overline{\mathbf{a}},\overline{\mathbf{u}})$, that minimizes $f$ and $f$ is convex in the neighborhood of this minimizer.
% Of course, this requires $f$ to vary continuously in $\mathbf{a}$.
% Note, however, that $f$ may vary continuously in $\mathbf{a}$ even if the elements of $W$ do not! All we need is for $W(\mathbf{a})\mathbf{u}$ to vary continuously in $\mathbf{a}$.
%
% Simon: I think the argument below is meaningless because $W(\mathbf{a})\mathbf{u}$ varies continuously with in a for arbitrary u if and only if all entries of W(a) vary continuously
% in a (and we don't want the scheme to be continuous just for SOME object u, right?).
% To my mind, however the assumption of local convexity need not be related with continuity here, anyway.
%
\end{itemize}
We note that only the first assumption is absolutely critical, but can be fulfilled easily in practice. The second assumption is necessary for some of the results that follow.
For the remainder of the section we will furthermore assume that
\begin{itemize}
\item[(A3)] $f$ is  twice continuously differentiable in $\mbf{a}$. This follows immediately if $W$ is twice continuously differentiable in $\mathbf{a}$ which requires the use of higher-order quadrature rules for approximating the projection-integrals.
% (e.g., using piecewise cubic polynomials)
%
% Simon: the problem is that piecewise cubic (Hermite) interpolation, as we use, is indeed not C^2 (but only C^1 with Lipschitz-continuous gradients)
%
\item[(A4)] $h$ is twice continuously differentiable.
\end{itemize}
We will discuss to what extent these latter assumptions can be dropped in section \ref{SS:lessDiff}.

The gradient and Hessian of $f$ are given by
\begin{equation}
\nabla f(\mbf{a},\mbf{u}) = \left(
\begin{matrix}
G(\mbf{a},\mbf{u})^T\!\left(W(\mathbf{a})\mathbf{u} - \mathbf{p}\right)\\
W(\mbf{a})^T\!\left(W(\mathbf{a})\mathbf{u} - \mathbf{p}\right) + \nabla h(\mathbf{u})
\end{matrix}
\right),
\end{equation}
and
\begin{equation}
\nabla^2 f(\mbf{a},\mbf{u}) = \left(
\begin{matrix}
G(\mathbf{a},\mathbf{u})^TG(\mathbf{a},\mathbf{u}) + R(\mathbf{a},\mathbf{u},\mathbf{r}) & G(\mathbf{a},\mathbf{u})^TW(\mathbf{a}) + H(\mathbf{a},\mathbf{r})^T        \\
W(\mathbf{a})^TG(\mathbf{a},\mathbf{u}) + H(\mathbf{a},\mathbf{r}) & W(\mathbf{a})^TW(\mathbf{a}) + \nabla^2 h(\mathbf{u})\\
\end{matrix}
\right),
\end{equation}
where $G(\mbf{a},\mbf{u})$ is the Jacobian of $\mbf{a} \mapsto W(\mathbf{a})\mathbf{u}$ with columns $(\partial W  / \partial a_i)\mathbf{u}$, $H(\mathbf{a},\mathbf{v})$ is the Jacobian of $\mbf a \mapsto W(\mathbf{a})^T\mathbf{v}$ with columns $\left( {\partial W}/{\partial a_i}\right)^T\mathbf{v}$,  $\mathbf{r} := W(\mathbf{a})\mathbf{u} - \mathbf{p}$ and $R(\mathbf{a},\mathbf{u},\mathbf{v}) \in \mR^{\ell \times \ell}$ has entries $\mathbf{u}^T ( \partial^2 W (\mbf a) / \partial a_i \partial a_j )\mathbf{v}$.
%As not all the entries of $W$ depend on all the alignment parameters, the sensitivity matrices exhibit a block-structure.
Note that by (A2), the full Hessian is positive semidefinite close to the minimizer.

We could in principle view (\ref{eq:joint}) as an optimization problem in $(\mathbf{a},\mathbf{u})$ and apply a gradient-based method to solve it directly. Such a brute-force joint optimization approach introduces several problems, however. Firstly, the problem can be severely ill-conditioned due to differences in sensitivity of $\mathbf{u}$ and $\mathbf{a}$. Furthermore, a joint approach makes it hard to employ advanced reconstruction techniques that have been developed for linear tomographic problems.

\subsection{Variable projection}

The joint estimation problem (\ref{eq:joint}) has a separable structure; it is strongly convex in $\mbf{u}$ and differentiable in $\mbf{a}$ (by assumption (A3)). Following the \emph{variable projection} approach \cite{golub73}, we eliminate the reconstructed object by setting $\overline{\mbf{u}}(\mbf{a}) := \argmin_{\mathbf{u}} f(\mathbf{a},\mathbf{u})$. The problem now reduces to
\begin{equation}
\min_{\mbf{a}} \overline{f}(\mbf{a}), \label{eq:ReducedMinProb}
\end{equation}
where $\overline{f}(\mathbf{a}) = f(\mbf{a},\overline{\mbf{u}}(\mbf{a}))$.

We can now state the following results:
\begin{theorem}
\label{thm:gradient}
Let $\overline{f}(\mbf{a}) = \argmin_{\mathbf{u}} f(\mathbf{a},\mathbf{u})$. The gradient of $\overline{f}$ is given by
\begin{eqnarray*}
\nabla \overline{f}(\mbf{a}) &=& \nabla_{\mathbf{a}} f(\mathbf{a},\overline{\mathbf{u}}),
\end{eqnarray*}
i.e. it is simply the gradient of the full objective, evaluated at the optimal $\mathbf{u}$. The Hessian of $\overline{f}(\mbf{a})$ is given by
\begin{equation}
\nabla^2 \overline{f}(\mbf{a}) = \nabla^2 f(\mathbf{a},\overline{\mathbf{u}}) \backslash (W^TW + \nabla^2 h(\overline{\mathbf{u}})),
\end{equation}
i.e., the Schur complement of $(W^TW + \nabla^2 h)$ in the full Hessian of $f$.
\end{theorem}
\begin{proof}
The gradient-expression follows directly from \cite[Thm.\ 2]{Bell} which requires assumption A1, while the Schur-complement expression for the Hessian was presented by \cite{AxelRuheandPerAkeWedin1980a}.
\end{proof}

\begin{remark}
By Theorem \ref{thm:gradient} and assumption (A2) we find that $\overline{f}$ is locally convex as well, according to the properties of the Schur complement. In addition, we find that a (local) minimum $\overline{\mbf{a}}$ of $\overline{f}$ together with the corresponding $\overline{\mathbf{u}}$ is a (local) minimum of $f$.
\end{remark}

The reduced problem \eqref{eq:ReducedMinProb} has a much smaller dimensionality and is expected to be less ill-conditioned. Indeed, we find the following relation between the condition numbers of the Hessians of $f$ and $\overline{f}$.

\begin{theorem}
\label{thm:conditioning}
Under assumptions A1-A4 we find the following relation between the condition numbers of the full and reduced Hessians locally around a minimizer:
\[
\kappa\left(\nabla^2\overline{f}\right) \leq \kappa\left(\nabla^2 f\right),
\]
where $\kappa$ denotes the condition number of the corresponding matrix.
\end{theorem}
\begin{proof}
Following \cite[Thm.\ 5]{Smith1992}, we find that the eigenvalues of $\nabla^2\overline{f}$ interlace those of $\nabla^2 f$. The result immediately follows.
\end{proof}

A more detailed analysis of the expected convergence behaviour when minimizing separable functions like $f$ and its reduced variant, $\overline{f}$, with Newton-like methods is presented by \cite{AxelRuheandPerAkeWedin1980a}.

\subsection{Inexactness}
In practice, it is not feasible to compute $\overline{\mathbf{u}}$ such that $\nabla_{\mathbf{u}}f(\mathbf{a},\overline{\mathbf{u}}) = 0$ exactly. Rather, let us assume we have an approximation $\overline{\mathbf{u}}_{\epsilon}$ such that
\[
\|\nabla_{\mathbf{u}}f(\mathbf{a},\overline{\mathbf{u}}_{\epsilon})\|_2 \leq \epsilon,
\]
and define
\[
\overline{f}_{\epsilon}(\mbf{a}) = f(\mbf{a},\overline{\mathbf{u}}_{\epsilon}).
\]
Although we could in principle compute the \emph{exact} gradient of $\overline{f}_{\epsilon}$ by taking into account the finite accuracy of the reconstruction, this would require us to differentiate the algorithm used to obtain $\overline{\mathbf{u}}_{\epsilon}$. This is not very attractive from a practical point-of-view. Instead, we approximate the gradient as $\nabla\overline{f}_{\epsilon}(\mbf{a}) \approx \nabla\overline{f}(\mbf{a}) = \nabla_{\mbf{a}}\overline{f}(\mbf{a},\overline{\mathbf{u}}_{\epsilon})$ and analyze the error we introduce. Under assumptions (A1)-(A4), we obtain the following result.

\begin{theorem}
\label{thm:inexact}
Given $\overline{\mathbf{u}}_{\epsilon}$ satisfying $\|\nabla_{\mathbf{u}}f(\mathbf{a},\overline{\mathbf{u}}_{\epsilon})\|_2 \leq \epsilon$, define the approximate gradient
\[
\nabla \overline{f}_{\epsilon}(\mathbf{a}) = \nabla_{\mathbf{a}} f(\mathbf{a},\overline{\mathbf{u}}_{\epsilon}).
\]
We then have
\[
\|\nabla \overline{f}_{\epsilon}(\mathbf{a}) - \nabla \overline{f}(\mathbf{a})\|_2 \leq C\epsilon,
\]
for some $C > 0$.
\end{theorem}

\begin{proof}
The approximate gradient is readily expressed as
\[
\nabla \overline{f}_{\epsilon}(\mathbf{a}) = \nabla \overline{f}(\mathbf{a}) + \nabla^2_{\mathbf{a},\mathbf{u}} f \left(\nabla^2_{\mathbf{u},\mathbf{u}} f\right)^{-1}\nabla_{\mathbf{u}} f(\mathbf{a},\overline{\mathbf{u}}_{\epsilon}),
\]
from which we find
\[
\|\nabla \overline{f}_{\epsilon}(\mathbf{a}) - \nabla \overline{f}(\mathbf{a})\|_2 \leq C\epsilon,
\]
where $C = \|\nabla^2_{\mathbf{a},\mathbf{u}} f \left(\nabla^2_{\mathbf{u},\mathbf{u}} f\right)^{-1}\|_2$.
\end{proof}
Gradient-descent algorithms with approximate derivatives have been studied extensively, and there are convergence guarantees to a local minimum of $\overline{f}$ as long as $\epsilon \downarrow 0$ \emph{or} to a point within distance $\mathcal{O}(\epsilon)$ of the local minimum for a fixed $\epsilon$ \cite{DAspremont2008,Friedlander2012}.

\subsection{Extensions to less-differentiable cases} \label{SS:lessDiff}
We can drop the requirement for the existence of second derivatives (A3) if we instead assume that the gradient of $f$ w.r.t.\ $(\mathbf{a},\mathbf{u})$ is Lipschitz continuous. The resulting reduced objective $\overline{f}$ will have a Lipschitz continuous gradient as well. Results analogous to Theorems \ref{thm:conditioning} and \ref{thm:inexact} can then be derived using the local convexity of $f$ (A2).
%\[
%\nabla f(\mathbf{z})^T(\mathbf{z} - \mathbf{z}') + \frac{1}{2}\|\mathbf{z} - \mathbf{z}'\|_L^2 \leq f(\mathbf{z}) - f(\mathbf{z}') \leq \nabla f(\mathbf{z})^T(\mathbf{z} - \mathbf{z}') + \frac{1}{2}\|\mathbf{z} - \mathbf{z}'\|_U^2,
%\]
%where $\|\mathbf{z}\|_{L} = \mathbf{z}^TL\mathbf{z}$ and
%$L$ and $U$ are both positive-definite block matrices. In the twice-differentiable case, these matrices would satisfy $L\preceq \nabla^2f\preceq U$.

Of special interest are non-differentiable regularization terms on $\mathbf{u}$ such as Total Variation (TV) or bound-constraints. We can readily drop assumption (A4) because the assumption of strong convexity of $f$ in $\mathbf{u}$ (A1) is sufficient to guarantee differentiability of $\overline{f}(\mathbf{a})$, see \cite[Thm.\ 9.18]{rockafellar_variational_2009}.

In many practical settings, the evaluation of $W$ is implemented via (multi-)linear interpolation \cite{xu_comparative_2006}, in which case the derivatives with respect to $\mathbf{a}$ are not continuous. Under assumption (A2) that $f$ is (locally) convex, the reduced function $\overline{f}$ is (locally) convex as well \cite[Thm.\ 2.22]{rockafellar_variational_2009} so that we can use a deritative-free method for minimizing $\overline{f}(\mathbf{a})$ in this setting. A promising candidate for such problems is the implicit filtering method, which approximates the derivative using finite-differences and adapts the stepsize according to its convergence \cite{Nocedal2006}.

\section{Algorithms}
\label{algorithm}
Based on the ideas presented in the previous section, we can design and analyze various joint alignment and reconstruction algorithms that aim to solve
\begin{equation}
\label{eq:reduced}
\min_{\mbf{a}} \overline{f}(\mbf{a}) + g(\mbf{a}),
\end{equation}
where $g(\mbf{a})$ is an appropriate regularization term. Below, we present a few prototype gradient-based algorithms for solving \eqref{eq:reduced} and discuss their convergence behaviour. To facilitate this discussion, we assume that $\overline{f}$ has a Lipschitz continuous gradient (which requires Lipschitz continuity of the gradient of $f$ w.r.t $\mathbf{a}$) and that $\overline{f} + g$ is strongly convex, i.e.,
\[
\frac{\mu}{2}\|\mbf{a} - \mbf{a}'\|_2^2 \leq \overline{f}(\mbf{a}) + g(\mbf{a}) - \overline{f}(\mbf{a}') - g(\mbf{a}')
\]
for some $\mu > 0$. As noted before, we do not expect $\overline{f}$ to be globally (strongly) convex, but it is reasonable to assume $\overline{f}$ to be locally (strongly) convex to facilitate the exposition. We can typically drop the assumption of strong convexity to mere convexity at the cost of a slower convergence rate. When dropping the requirement of convexity altogether, we are left with a non-linear optimization problem without any special structure and cannot state any convergence results.

\subsection{Joint reconstruction and alignment}
Based on the properties of the regularization term $g$, we can design a rule to update the alignment parameters.
A basic first-order algorithm for a smooth regularization term $g$ is shown in Algorithm~\ref{algorithm:smooth}.
\begin{algorithm}
\caption{Basic first order algorithm for solving the reduced optimization problem.}
\label{algorithm:smooth}
\begin{algorithmic}
\REQUIRE{Initial alignment parameters: $\mathbf{a}^{(0)}$, stepsizes: $\gamma^{(k)}$, tolerances: $\epsilon_k$}
\ENSURE{Final reconstruction and alignment parameters: $\mathbf{u}^{(k)}$, $\mathbf{a}^{(k)}$}
\STATE{$k = 0$}
\WHILE{not converged}
\STATE{$\mathbf{u}^{(k+1)} = \mathsf{reconstruct}(W(\mathbf{a}^{(k)}),\mathbf{p},\epsilon_k)$}
\STATE{$\mathbf{a}^{(k+1)} = \mathbf{a}^{(k)} - \gamma^{(k)}\left(G(\mathbf{a}^{(k)},\mathbf{u}^{(k+1)})^T\left(W(\mathbf{a}^{(k)})\mathbf{u}^{(k+1)} - \mathbf{p}\right) + \nabla g(\mbf{a})\right)$}
\STATE{$k = k + 1$}
\ENDWHILE
\end{algorithmic}
\end{algorithm}

Here and in the following, the operation $\mathsf{reconstruct}(W(\mathbf{a}^{(k)}),\mathbf{p},\epsilon_k)$  corresponds to solving the tomographic reconstruction problem at fixed alignment parameters
\begin{equation*}
 \min_{\mbf u } f(\mbf u, \mathbf{a}^{(k)})
\end{equation*}
up to a tolerance $\epsilon_k$ in the gradient $\|\nabla_{\mbf{u}} f(\mathbf{a}^{(k)},\mbf{u}^{(k+1)})\|_2 \leq \epsilon_k$. For quadratic regularizers $h$, the minimization can be performed efficiently via a conjugate-gradient method. For non-smooth convex penalties involving for example total variation or bound-constraints, Chambolle-Pock's algorithm \cite{ChambollePock2011} or FISTA \cite{Beck2009Teboulle_FISTA} are suitable choices.

The following theorem establishes convergence of Algorithm~\ref{algorithm:smooth} to the minimizer, $\overline{\mathbf{a}}$, of $\overline{f}$ for a fixed steplength.
\begin{theorem} \label{thm:ConvResult}
Let $\gamma^{(k)} = 1/L$ and $\|\nabla_{\mbf{u}} f(\mathbf{a}^{(k)},\mbf{u}^{(k+1)})\|_2 \leq \epsilon_k$ hold in each iteration $k$ of Algorithm~\ref{algorithm:smooth}. Then we have for some constant $C>0$
\[
\left(\overline{f}(\mathbf{a}^{(k+1)}) - \overline{f}(\overline{\mathbf{a}})\right) \leq (1 - \mu/L)\left(\overline{f}(\mathbf{a}^{(k)}) - \overline{f}(\overline{\mathbf{a}})\right) + \frac{C^2\epsilon_k^2}{2 L},
\]
where $L$ denotes the Lipschitz constant of $\nabla\overline{f} + \nabla g$ and $\mu$ is the modulus of strong convexity of $\overline{f} + g$.
\end{theorem}
\begin{proof}
From theorem \ref{thm:inexact} we immediately obtain a bound of the error in the gradient. Applying \cite[Lemma 2.1]{Friedlander2012} gives the desired result.
\end{proof}

According to theorem \ref{thm:ConvResult}, Algorithm~\ref{algorithm:smooth} is linearly convergent under the given Lipschitz-continuity- and strong-convexity-assumptions if we decrease the tolerance $\epsilon_k$ at a linear rate (e.g., $\epsilon_k = r^k$ with $r < 1$) \cite{Friedlander2012}.
% Note that, in practice, we determine the stepsizes $\gamma^{(k)}$ via an appropriate linesearch procedure, see \ref{SS:alignMethod}.

When the regularization term $g$ is convex but not differentiable, we can apply a proximal gradient algorithm, as stated in Algorithm~\ref{algorithm:prox}. Here, the proximal operator is defined as
\[
	\mathsf{prox}_{\gamma g}(\mathbf{x}) = \argmin_{\mbf{a}} \textstyle{\frac{1}{2}}\|\mbf{a} - \mbf{x}\|_2^2 + \gamma g(\mbf{a})
\]
This can be used, for example, to impose bound-constraints, i.e.\ admissible minimum and maximum values, on the alignment parameters.
\begin{algorithm}
\caption{Basic proximal-gradient algorithm for solving the reduced optimization problem.}
\label{algorithm:prox}
\begin{algorithmic}
\REQUIRE{Initial alignment parameters: $\mathbf{a}^{(0)}$, tolerance: $\epsilon_k$}
\ENSURE{Final reconstruction and alignment parameters: $\mathbf{u}^{(k)}$, $\mathbf{a}^{(k)}$}
\STATE{$k = 0$}
\WHILE{not converged}
\STATE{$\mathbf{u}^{(k+1)} = \mathsf{reconstruct}(W(\mathbf{a}^{(k)}),\mathbf{p},\epsilon_k)$}
\STATE{$\mathbf{a}^{(k+1)} = \mathsf{prox}_{\gamma^{(k)} g}\left[\mathbf{a}^{(k)} - \gamma^{(k)}\left(G(\mathbf{a}^{(k)},\mathbf{u}^{(k+1)})^T\left(W(\mathbf{a}^{(k)})\mathbf{u}^{(k+1)} - \mathbf{p}\right)\right)\right]$}
\STATE{$k = k + 1$}
\ENDWHILE
\end{algorithmic}
\end{algorithm}

Convergence of proximal-gradient algorithms with approximate gradients has been established by \cite{Schmidt2011}. Provided that $\overline{f}$ is strongly convex and that the error in the gradient decreases at a linear rate, Algorithm~\ref{algorithm:prox} converges linearly. If $\overline{f}$ is only convex, we obtain a sublinear convergence rate.

A variant of the proposed algorithms is stated in Algorithm~\ref{algorithm:alternating}. The method $\mathsf{align}(\mathbf{u},\mathbf{p},\delta)$  computes minimizing alignment parameters $\argmin_{\mathbf{a}} f(\mathbf{a},\mathbf{u}^{(k)})$ up to some tolerance $\delta$.
% The separability of the problem can be exploited to reduce the alignment step to a series of low-dimensional minimizations (see \ref{SS:alignMethod}).
Convergence of Algorithm~\ref{algorithm:alternating} can be guaranteed in case $f(\mathbf{a},\mathbf{u})$ is convex in $\mathbf{a}$ as well as $\mathbf{u}$ \cite{Bleyer2013,Bleyer2015}. Convergence rates for this algorithm have, to the best of our knowledge, not been established.

\begin{algorithm}
\caption{Alternating minimization algorithm for solving the optimization problem.}
\label{algorithm:alternating}
\begin{algorithmic}
\REQUIRE{Initial alignment parameters: $\mathbf{a}^{(0)}$, tolerances: $\epsilon_k, \delta_k$}
\ENSURE{Final reconstruction and alignment parameters: $\mathbf{u}^{(k)}$, $\mathbf{a}^{(k)}$}
\STATE{$k = 0$}
\WHILE{not converged}
\STATE{$\mathbf{u}^{(k+1)} = \mathsf{reconstruct}(W(\mathbf{a}^{(k)}),\mathbf{p},\epsilon_k)$}
\STATE{$\mathbf{a}^{(k+1)} = \mathsf{align}(\mathbf{u}^{(k+1)},\mathbf{p},\delta_k)$}
\STATE{$k = k + 1$}
\ENDWHILE
\end{algorithmic}
\end{algorithm}
\begin{figure}[hbt!]
\centering
\includegraphics[scale=.5]{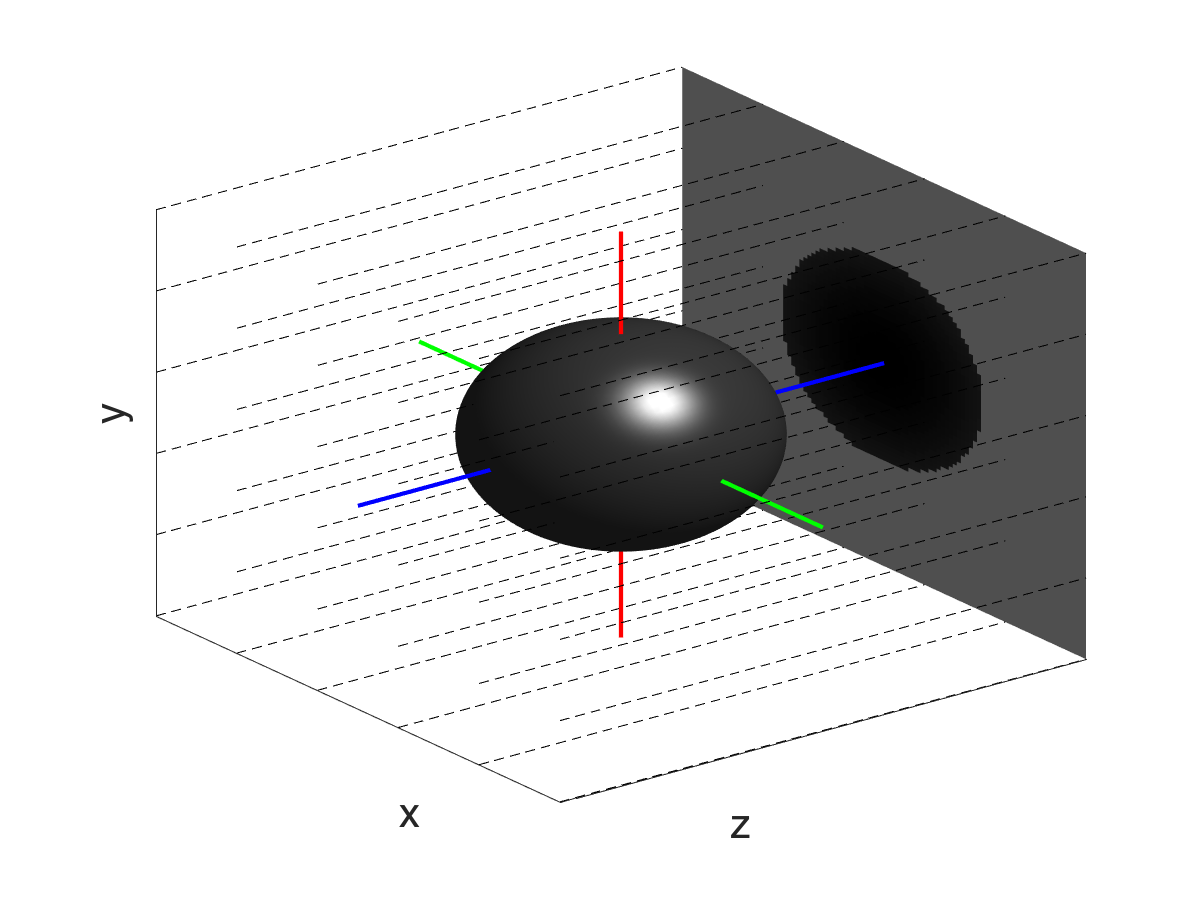}
\caption{Alignment parameters for the parallel-beam setup: \emph{in-plane} (around the $z$-axis), \emph{nod/pitch} (around the $x$-axis) and \emph{tomographic} (around the $y$-axis) rotations and \emph{lateral} (along the $x$-axis) and \emph{axial} (along the $y$-axis) shifts.}
\label{fig:alignment}
\end{figure}

\section{Results} \label{S:Results}
\label{results}
The following results have been obtained with a custom implementation of a three-dimensional projection operator for parallel-beam setups that allows for full three-dimensional alignment of all projection images independently. The alignment parameters are taken to be three rotations and two shifts relative to the direction of the beam, as shown in figure \ref{fig:alignment}. We refer to the rotations as \emph{in-plane} (around the $z$-axis), \emph{nod/pitch} (around the $x$-axis) and \emph{tomographic} (around the $y$-axis) and the shifts as \emph{lateral} (along the $x$-axis), \emph{axial} (along the $y$-axis). For the considered parallel beam setup, shifts along the optical axis (blue axis in figure \ref{fig:alignment}) do not influence the projections.
The tomographic projector is implemented in Matlab, exploiting its GPU capabilities to achieve computational efficiency. Implementation Details on the projection operators, the corresponding Jacobian, the reconstruction and alignment algorithms are described in the appendix. In all of the subsequent numerical experiments, we simulate the misaligned tomographic data using \emph{trilinear}
 interpolation, but applied the \emph{bicubic} scheme from \ref{SS:interpolation} in the joint alignment and reconstruction.

% Convergence experiments
\begin{figure}
\centering
\includegraphics[scale=1]{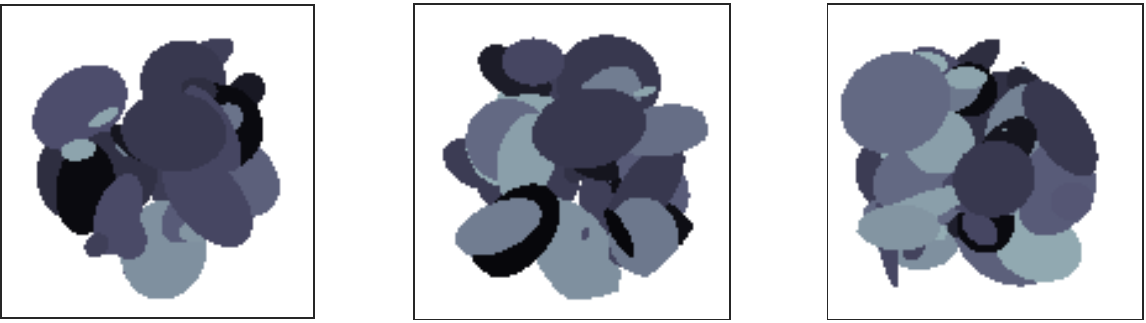}\\
\small{(phantom: central ortho-slices along the different coordinate planes)}\\
\includegraphics[scale=1]{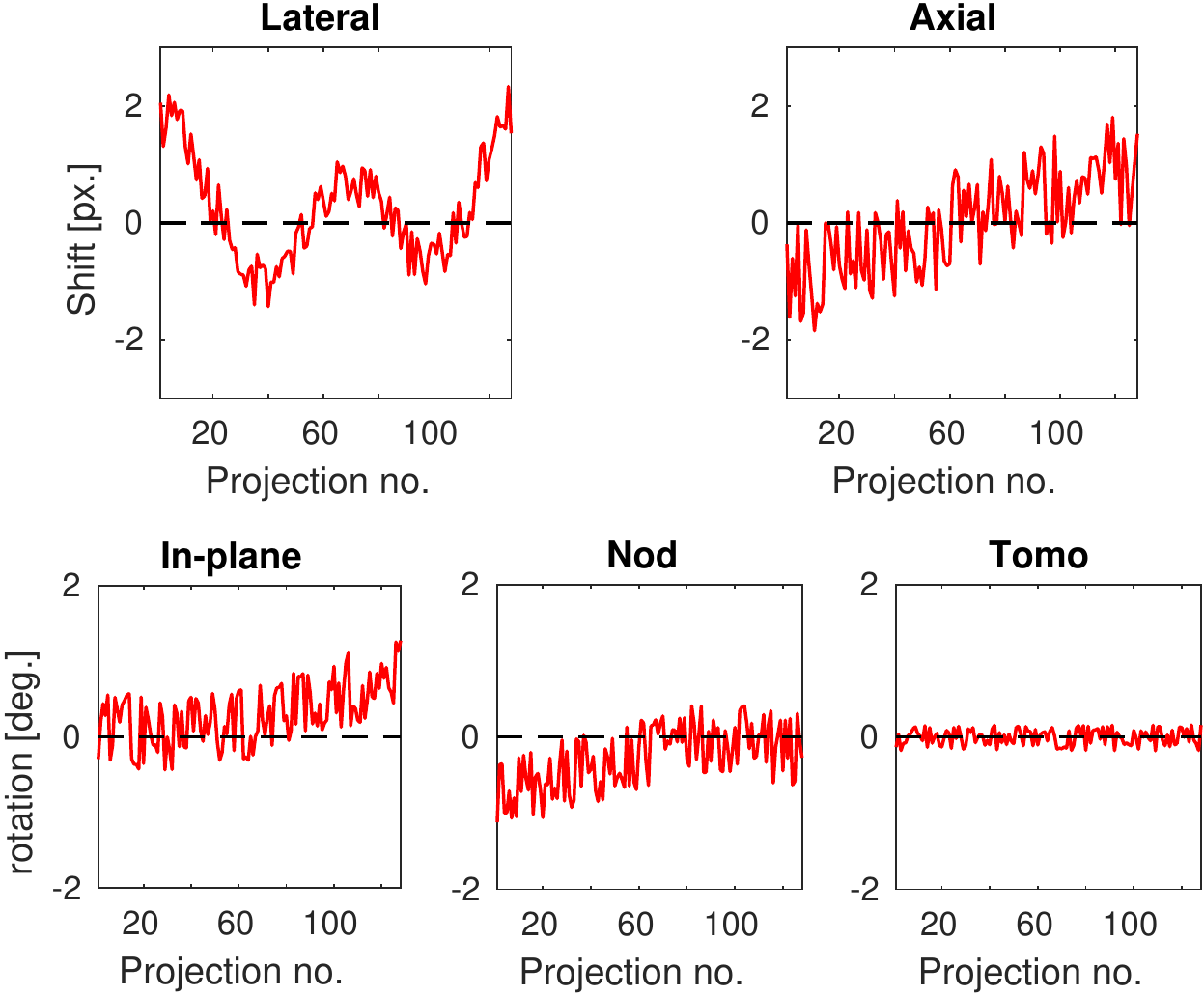}\\
\small{(simulated misalignment in the acquisition parameters: lateral- and axial shifts $s_x, s_y$ and in-plane-, nod- and tomographic rotations $\theta_{xy}, \theta_{yz}, \theta_{xz}$, compare figure~\ref{fig:alignment})}\\ % see \ref{SS:representAlign} for the details
\caption{Simulated phantom and perturbations of the alignment parameters in the convergence experiments of section \ref{SS:ConvExp}.}
\label{fig:phantom}
\end{figure}

\subsection{Convergence experiments} \label{SS:ConvExp}
We test algorithms \ref{algorithm:smooth} and \ref{algorithm:alternating} on a phantom of size $128\times 128\times 128$, shown in figure \ref{fig:phantom}. We consider a parallel-beam setup with 128 tomographic projections uniformly covering the full angular range $[0;\, 180^\circ]$. We simulate noise-free tomographic data under small perturbations of the alignment parameters that are shown in figure \ref{fig:phantom}. For tomographic reconstruction,  we use the Conjugate Gradient (CG) method to solve the gradient-penalized Tikhonov functional in \eqref{eq:TikhonovTomoRec}.
% \[
% \min_{\mathbf{u}} \textstyle{\frac{1}{2}}\|W(\mathbf{a})\mathbf{u} - \mathbf{p}\|_2^2  + \alpha \|L\mathbf{u}\|_2^2,
% \]
% where $L$ is the discrete Laplace operator.
The regularization paramater is fixed to $\alpha = 10^3$. In the experiments, the reconstruction tolerance $\epsilon$ (cf.\ Algorithm \ref{algorithm:smooth}) is a relative tolerance (relative to the initial norm of the gradient) and we control the alignment tolerance $\delta$ (cf.\ Algorithm \ref{algorithm:alternating}) by choosing the number of inner alignment gradient-descent iterations $N_{\text{align}}$. A large $N_{\text{align}}$ thus corresponds to a smaller $\delta$. for $N_{\text{align}} = 1$, algorithms \ref{algorithm:smooth} and \ref{algorithm:alternating} are obviously identical.

\begin{figure}
\centering
\includegraphics[scale=1]{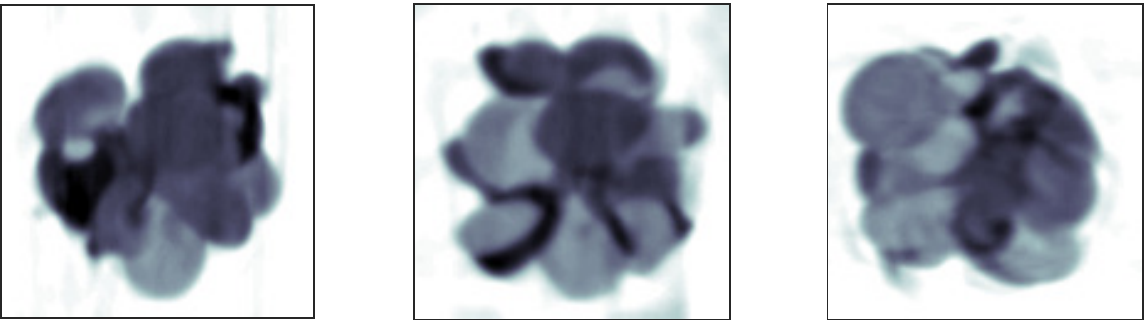}\\
\small{(un-aligned reconstruction: same ortho-slices as in figure \ref{fig:phantom})}\\
\includegraphics[scale=1]{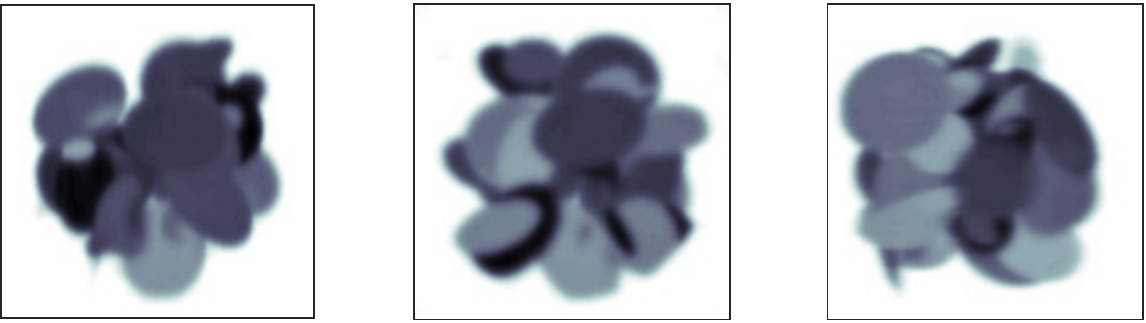}\\
\small{(aligned reconstruction for $\epsilon = 10^{-4}$, $N_{\text{align}} = 1$)}\\
\includegraphics[scale=1]{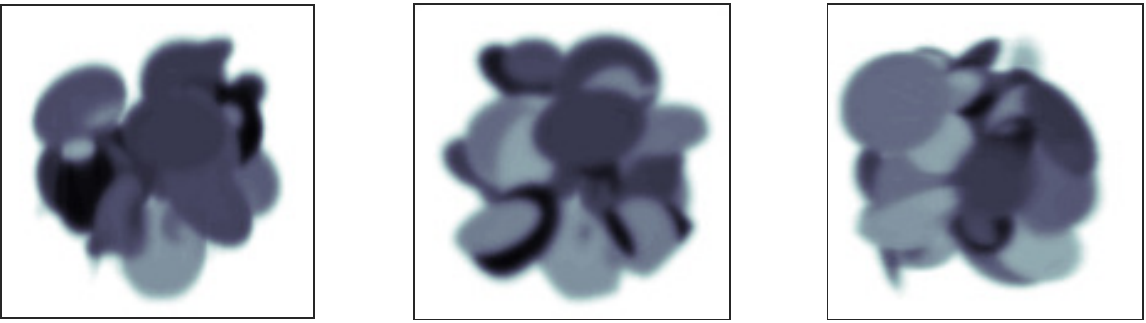}\\
\small{(reconstruction using the true alignment parameters)}\\
\caption{Reconstructed object from the initially guessed alignment parameters $\mbf{a}^{(0)}$ (un-aligned) and using the final reconstruction $\mbf{a}^{(50)}$ (aligned), respectively.
For comparison, we also show the reconstruction obtained using the true alignment parameters.}
\label{fig:reconstructions}
\end{figure}

Figure \ref{fig:reconstructions} shows the un-aligned and aligned reconstruction for $\epsilon = 10^{-4}$, $N_{\text{align}} = 1$, as well as the reconstruction with the true alignment parameters. The un-aligned reconstruction clearly shows blurring and streaks, which disappear in the aligned reconstruction. Moreover, the aligned reconstruction is visually indistinguishable from the one using the true alignment parameters.
The fitted alignment parameters are plotted in \ref{fig:alignments}. We see that the shifts are fairly well-recovered in all cases. However, the fit of the reconstruction of the rotational misalignment breaks down for the lowest CG-accuracy $\epsilon = 10^{-1}$. For smaller reconstruction-tolerances, all alignment parameters except for the tomographic angle are recovered reasonably well.
%, compare ``tomo''-plots in figure \ref{fig:reconstructions}.
The reconstruction of the latter remains corrupted by low-frequency errors, which may point to an inherent non-uniqueness in the problem.

\begin{figure}
\centering
\begin{tabular}{cc}
\includegraphics[scale=.5]{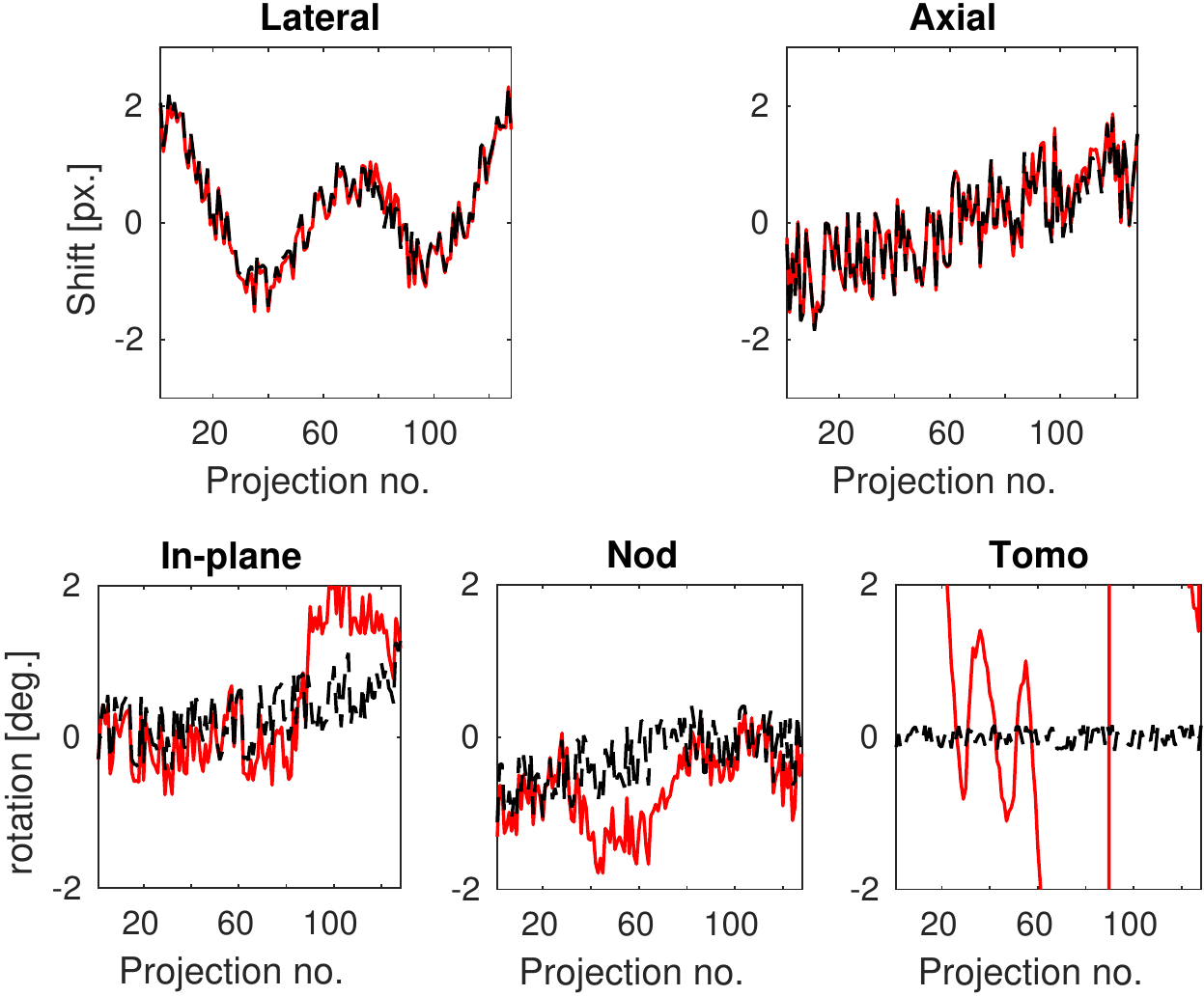}&
\includegraphics[scale=.5]{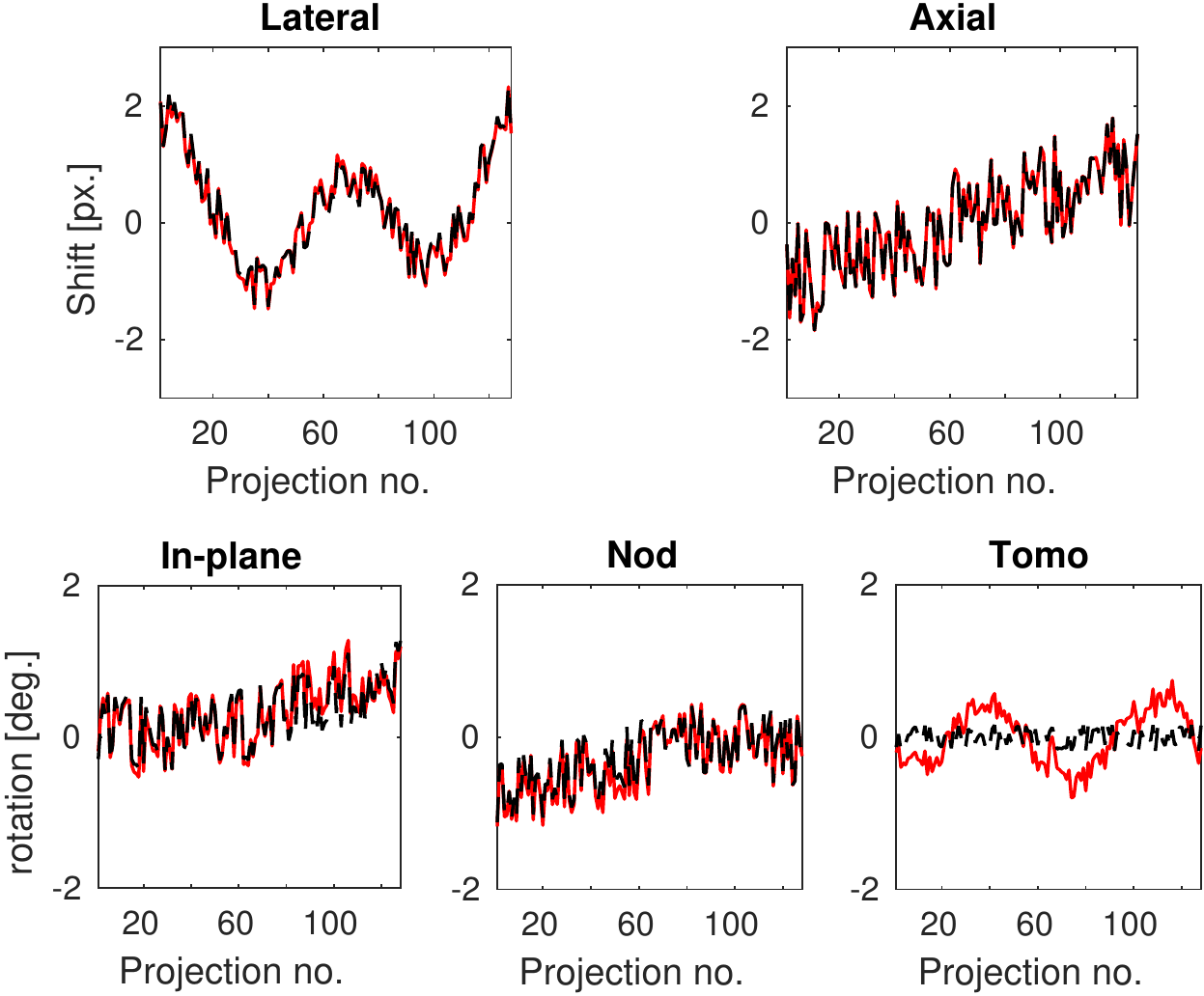}\\
\small{($\epsilon = 10^{-1}$)}&
\small{($\epsilon = 10^{-2}$)}\\
\includegraphics[scale=.5]{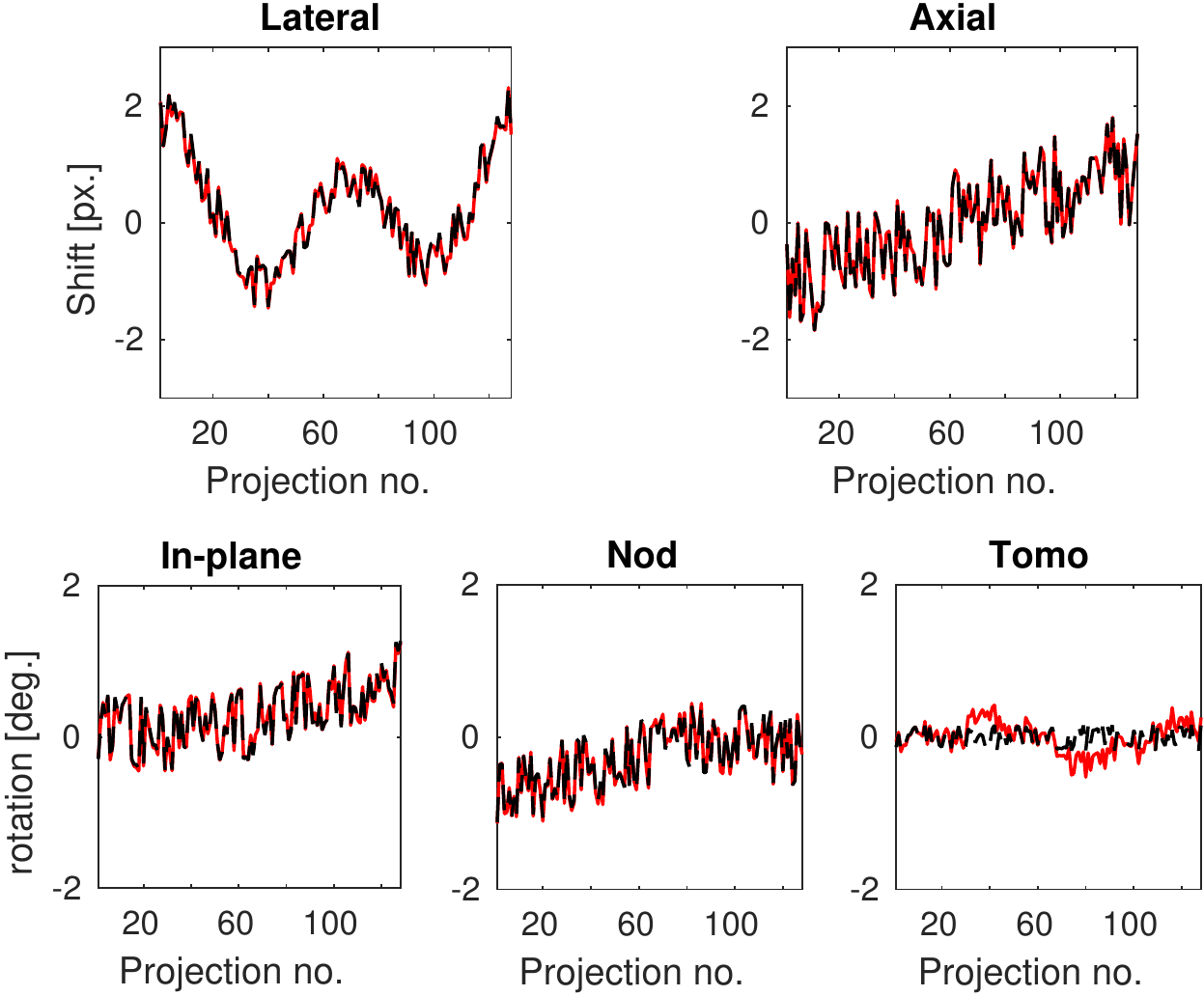}&
\includegraphics[scale=.5]{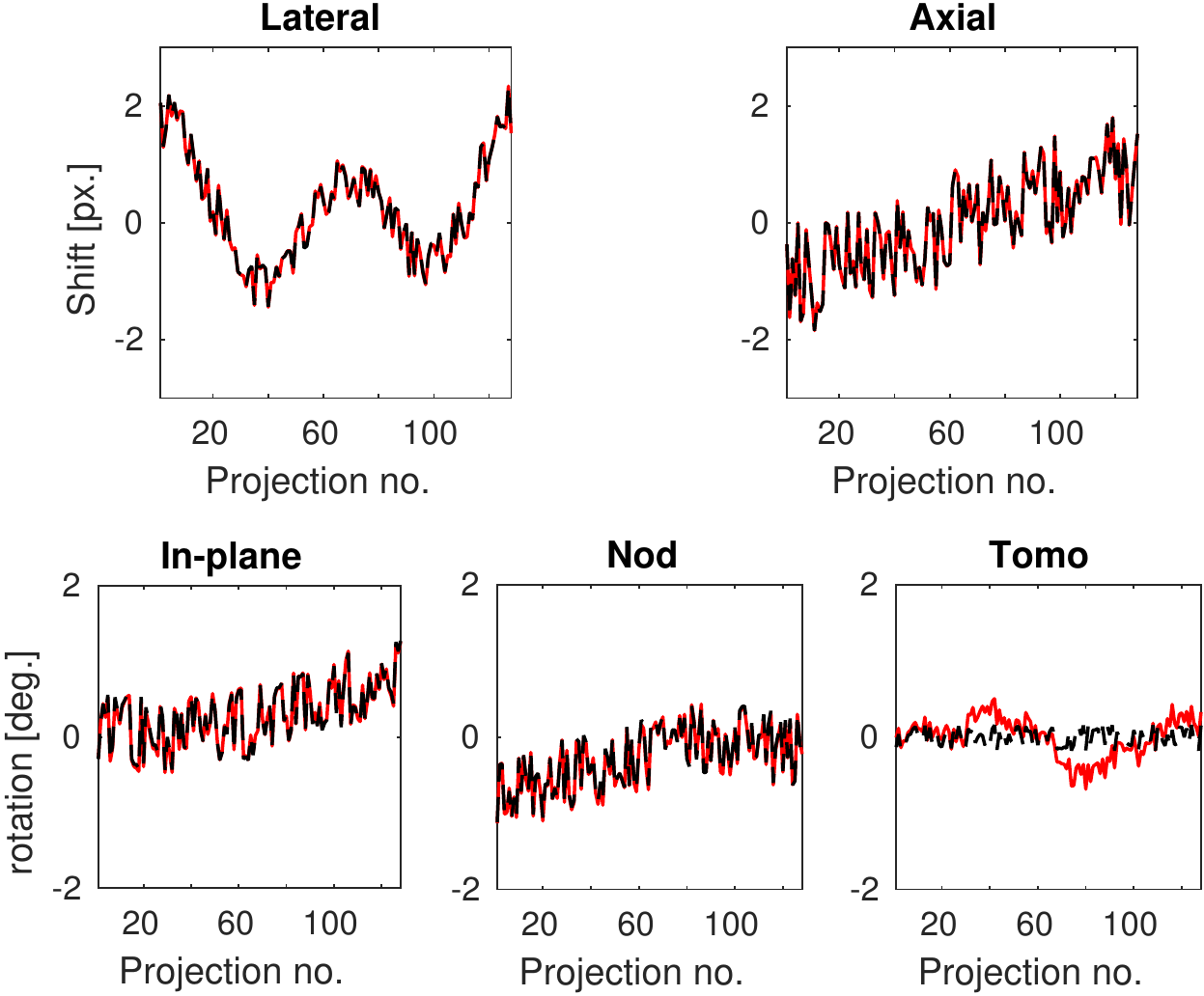}\\
\small{($\epsilon = 10^{-3}$)}&
\small{($\epsilon = 10^{-4}$)}\\
\end{tabular}
\caption{Alignment parameters reconstructed with Algorithm \ref{algorithm:smooth} after 50 iterations for different reconstruction tolerances $\epsilon$ of the CG-method.}
\label{fig:alignments}
\end{figure}

In figure \ref{fig:convergences}, we show the convergence of Algorithm \ref{algorithm:alternating} in terms of the optimality ($\|\nabla\overline{f}\|_2$), the object-reconstruction error and the alignment error (the difference between the true and reconstructed alignment parameters).
We note that using a larger reconstruction tolerance may lead to a premature stalling of the convergence, or even to divergence, compare results for  $\epsilon = 10^{-1}$.
Interestingly, a large $\epsilon$ also seems to result in a slightly faster decrease of the alignment error in the first few iterations. This might be explained by the fact that fewer CG-iterations lead to a smoother recontruction, which is easier to align as argued in \ref{SS:tomoReconMethod}. Furthermore, we observe that the alignment error (right column) always exhibits a slight increase after an initial phase of convergence. This is due to the emerging low-frequency errors in the fitted tomographic angle observed in figure \ref{fig:alignments}.
We also see that increasing the number of alignment steps ($N_{\text{align}}$) leads to slightly more optimal results (left column) but not to a higher accuracy in terms of the reconstruction and even gives a slightly larger alignment error. We may thus conclude that performing more than a single alignment update per iteration (cf.\ Algorithm \ref{algorithm:alternating}) is not necessarily better than Algorithm \ref{algorithm:smooth}.

% Remove bad paragraph-pagebreak behavior here:
\clubpenalty = 10000

\begin{figure}
\centering
\begin{tabular}{ccc}
\includegraphics[scale=.3]{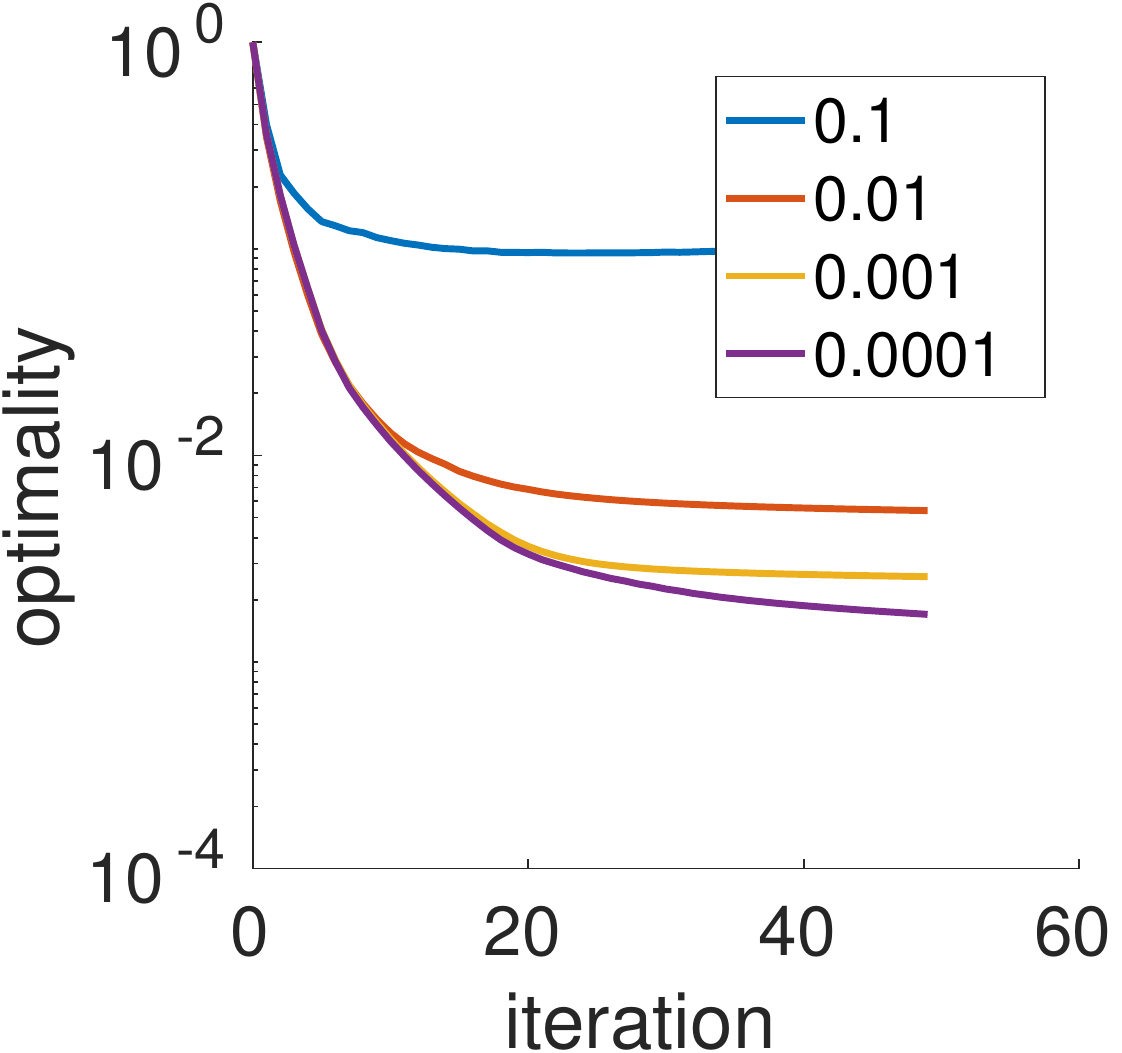}&
\includegraphics[scale=.3]{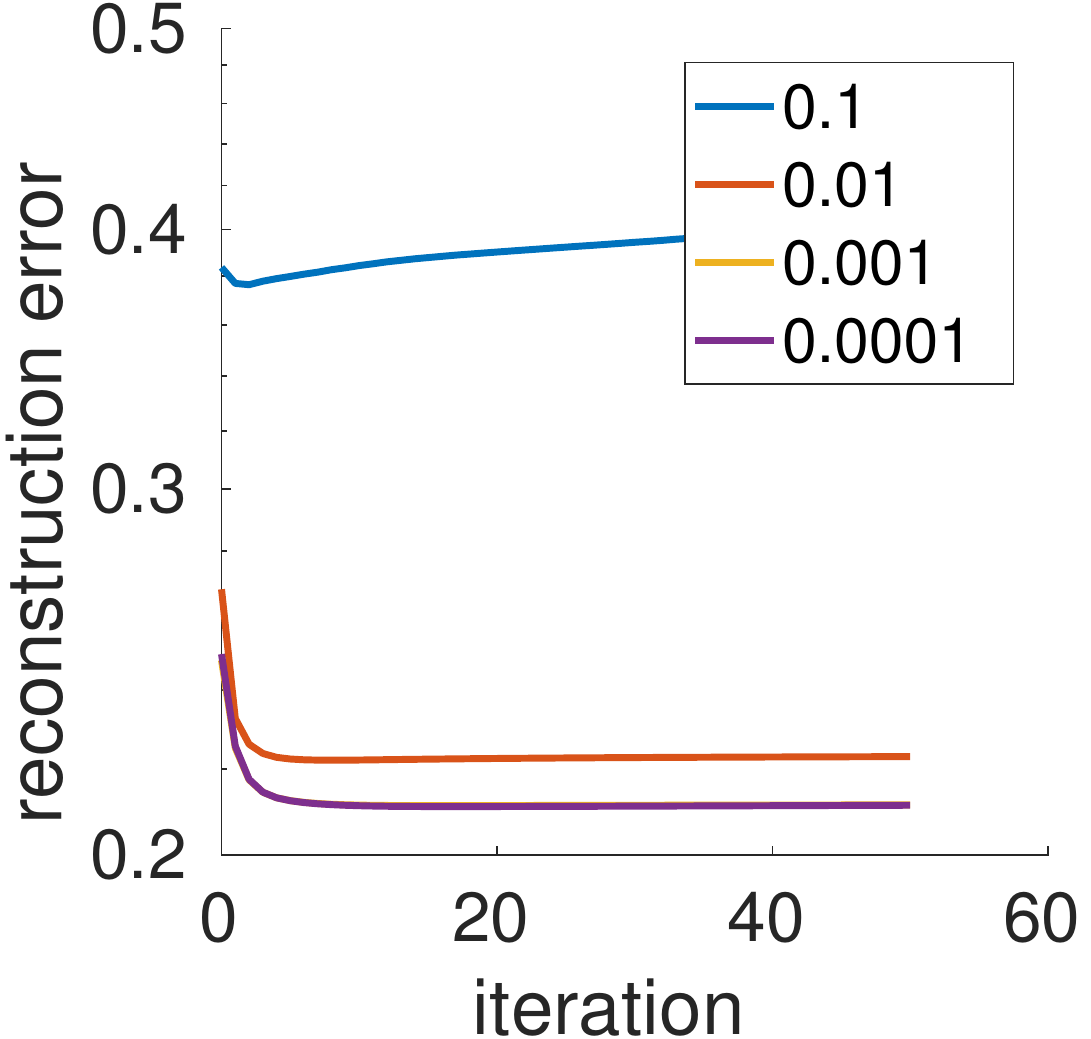}&
\includegraphics[scale=.3]{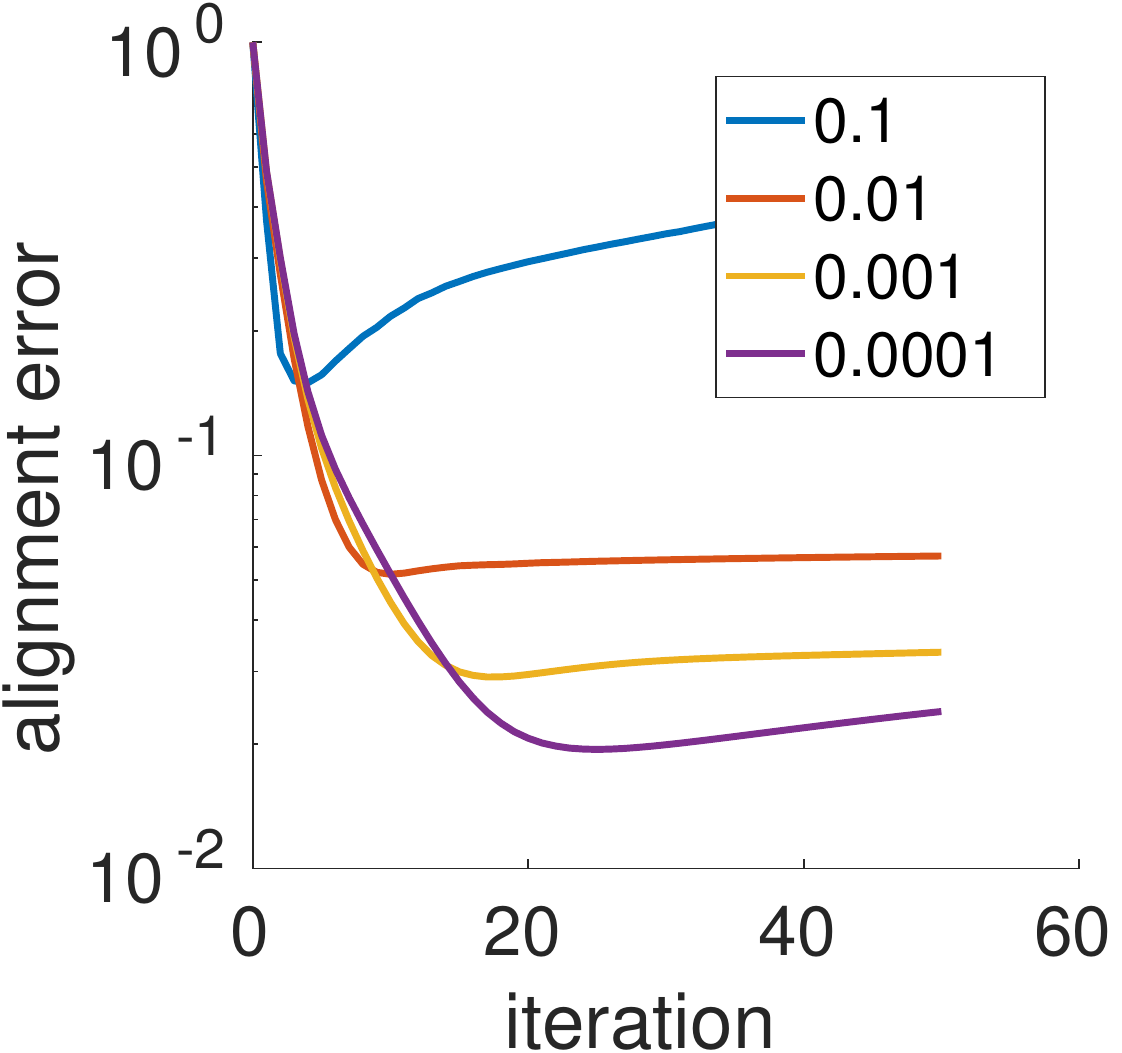}\\
&\small{$N_{\text{align}} = 1$}&\\
\includegraphics[scale=.3]{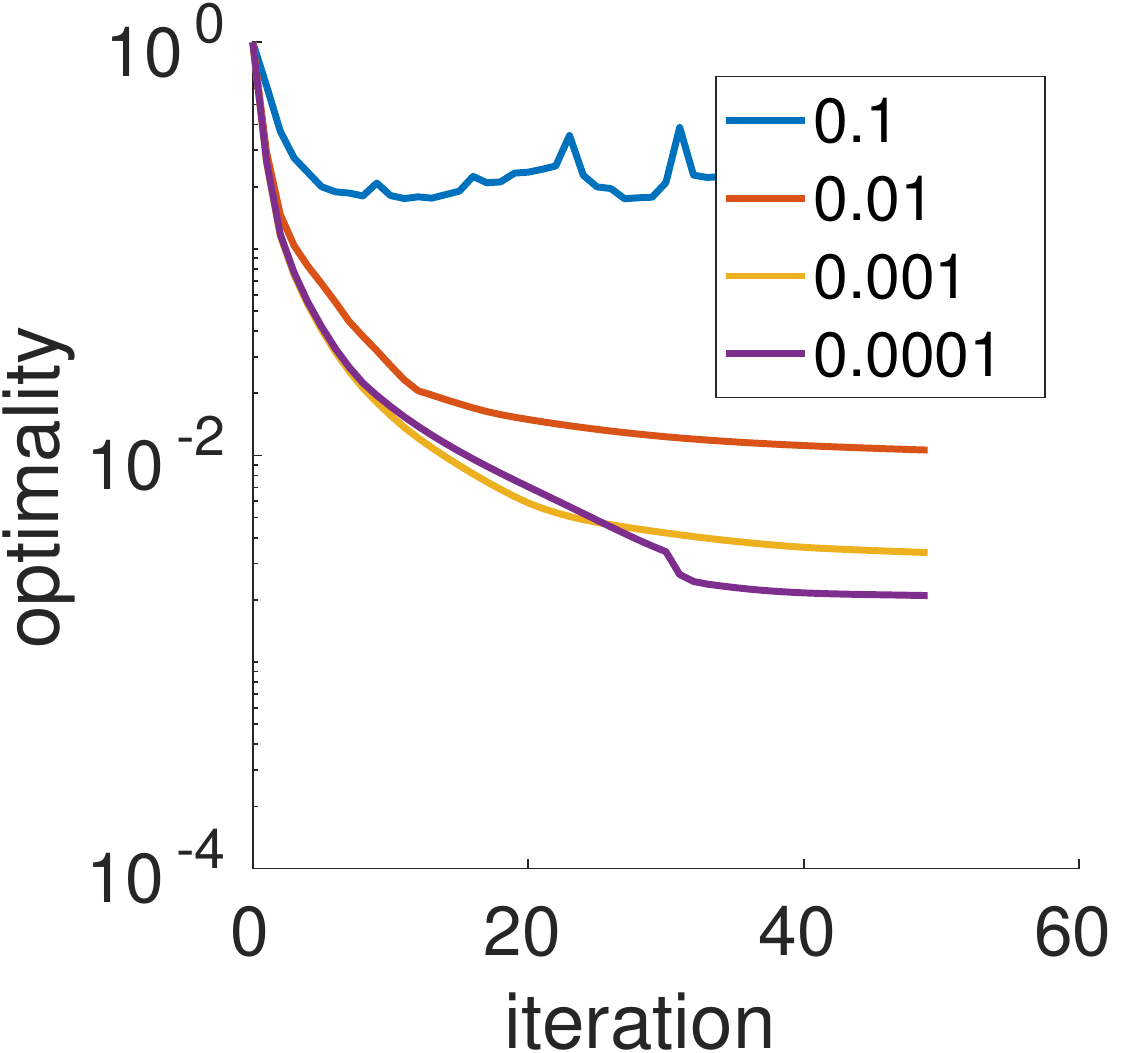}&
\includegraphics[scale=.3]{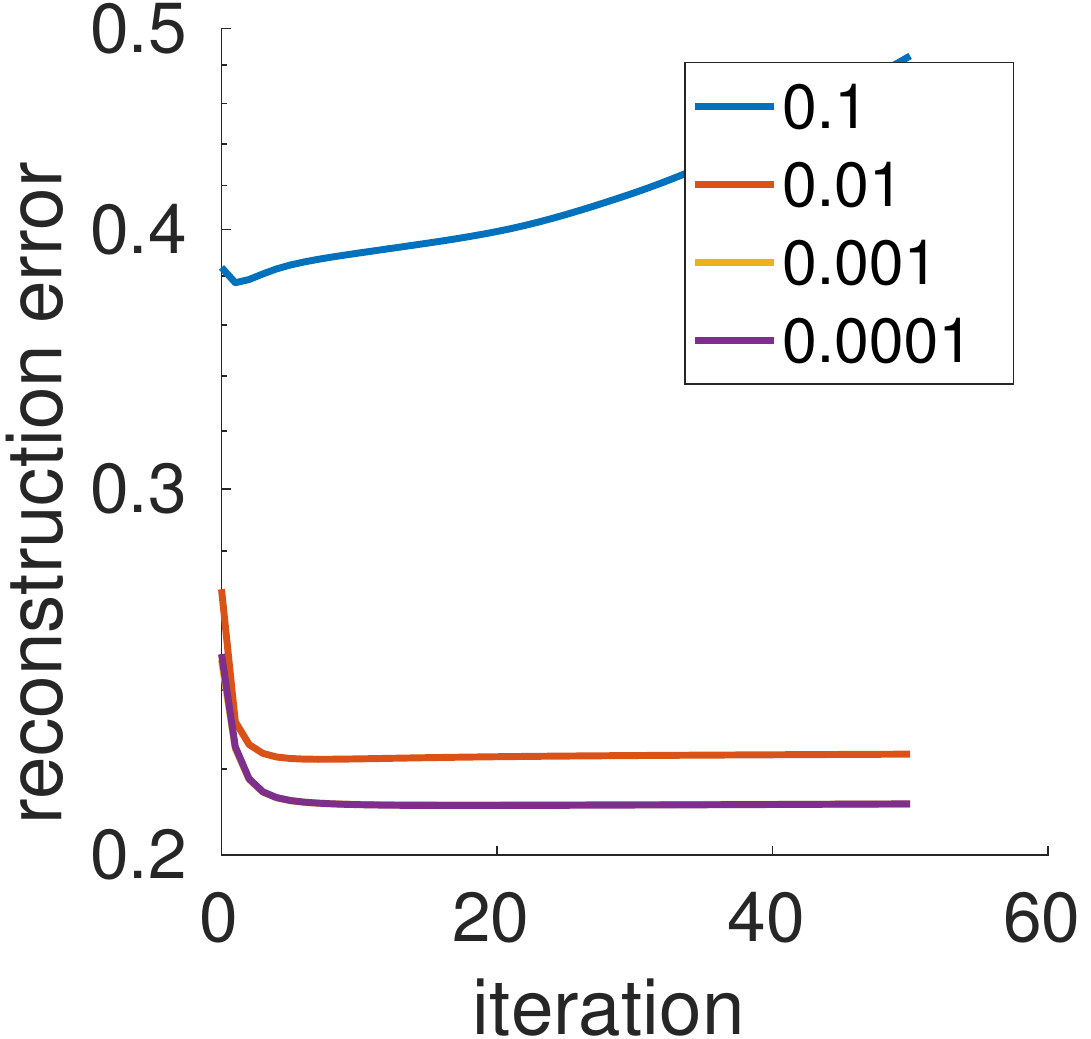}&
\includegraphics[scale=.3]{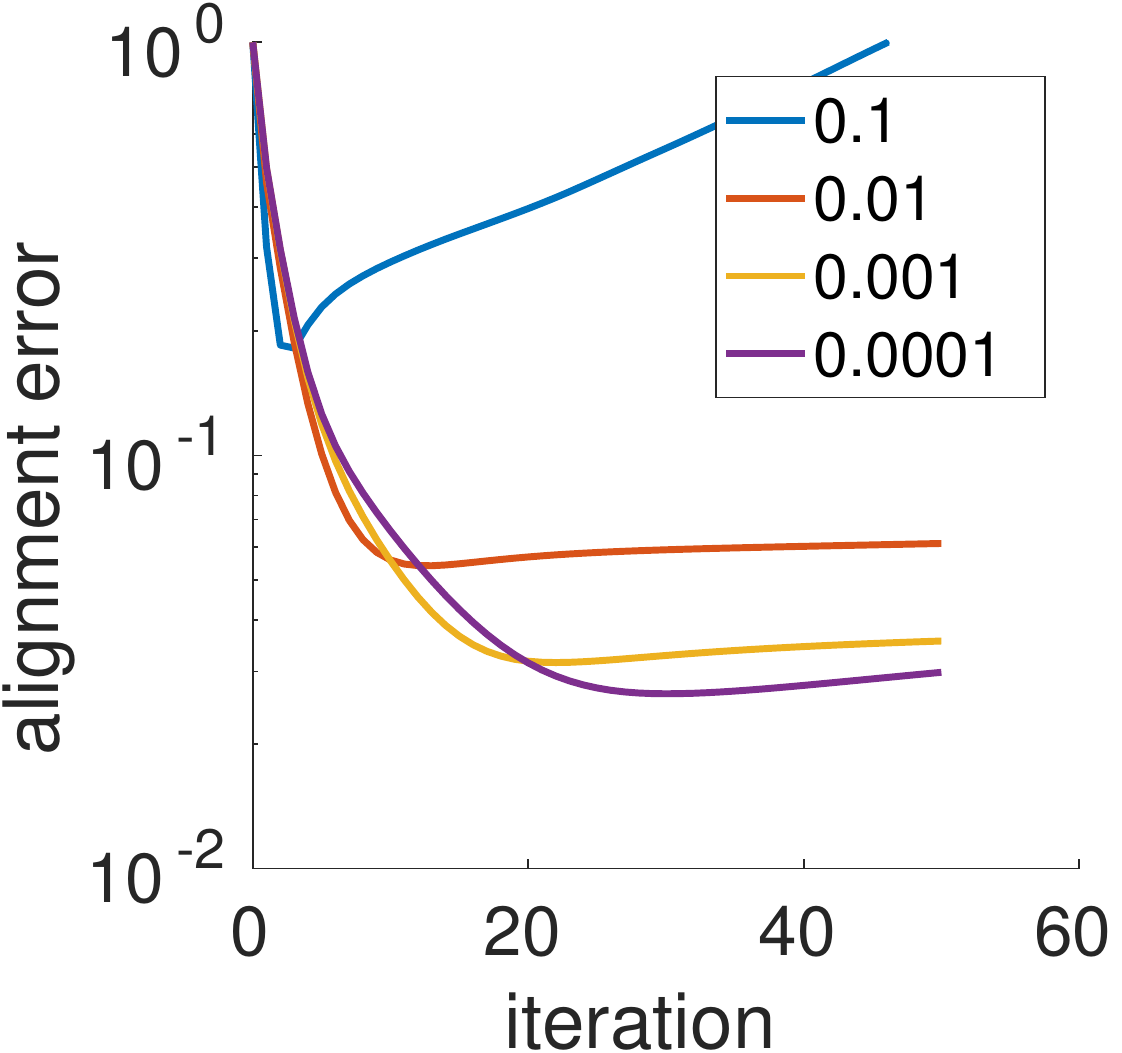}\\
&\small{$N_{\text{align}} = 2$}&\\
\includegraphics[scale=.3]{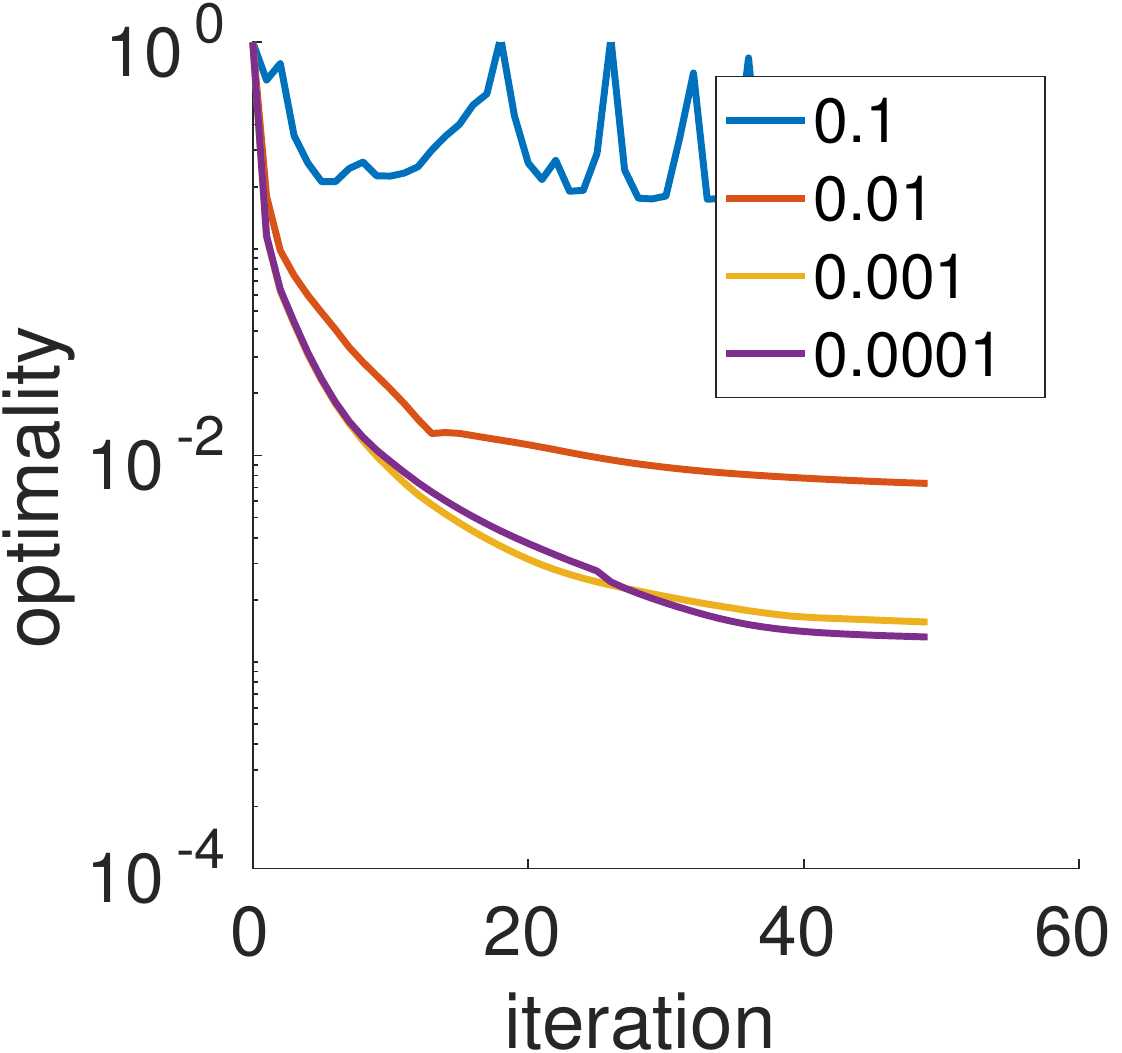}&
\includegraphics[scale=.3]{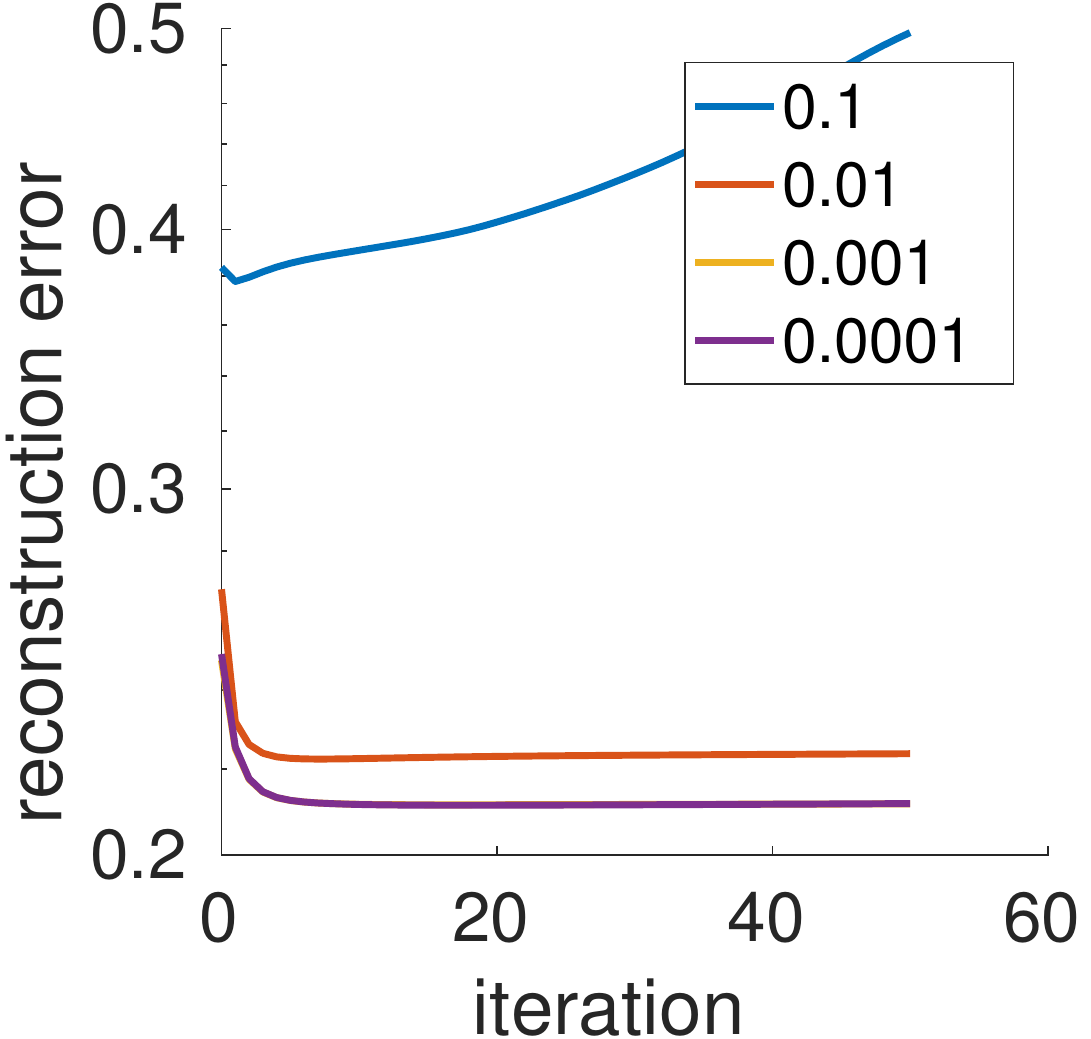}&
\includegraphics[scale=.3]{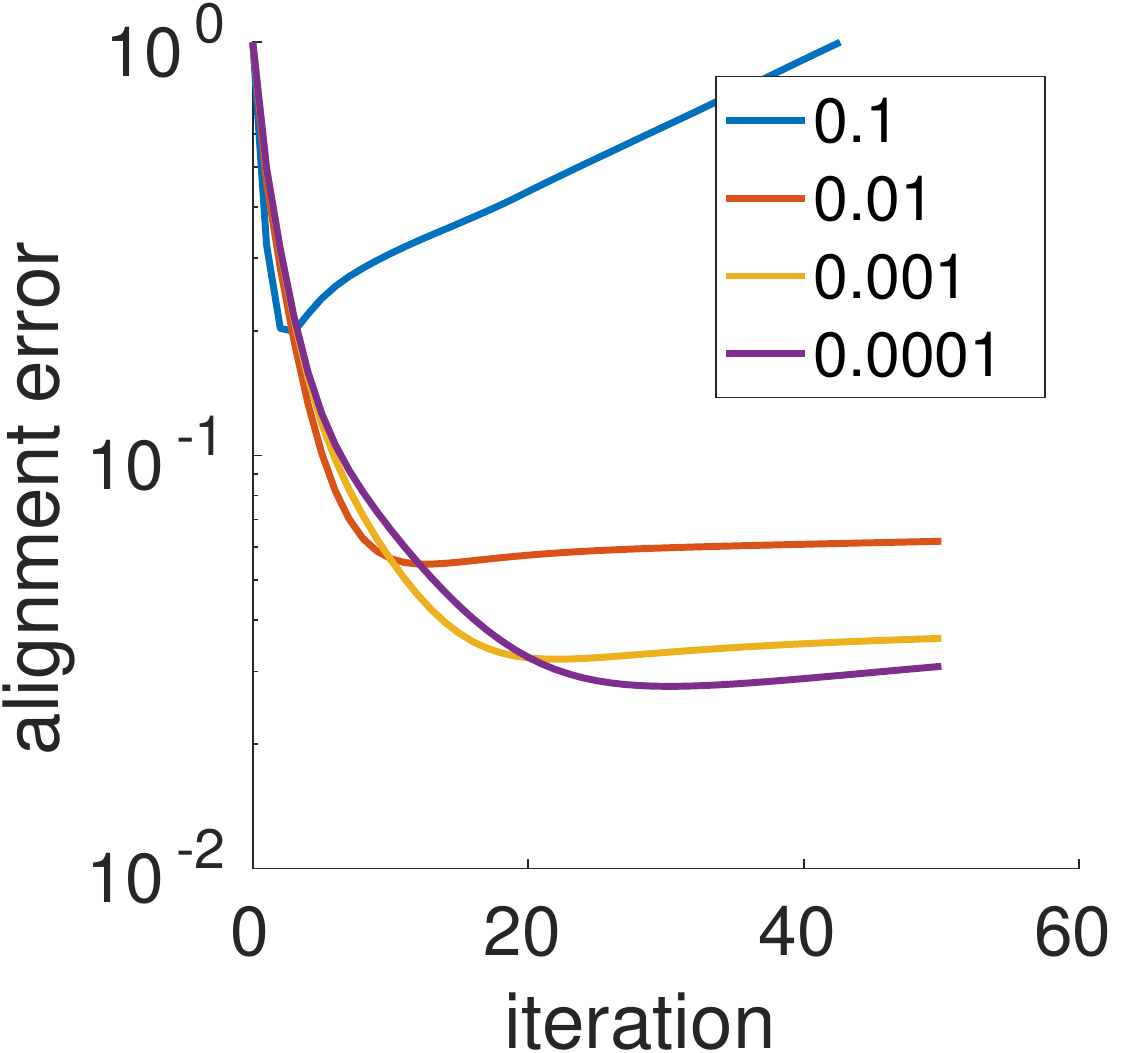}\\
&\small{$N_{\text{align}} = 4$}&\\
\includegraphics[scale=.3]{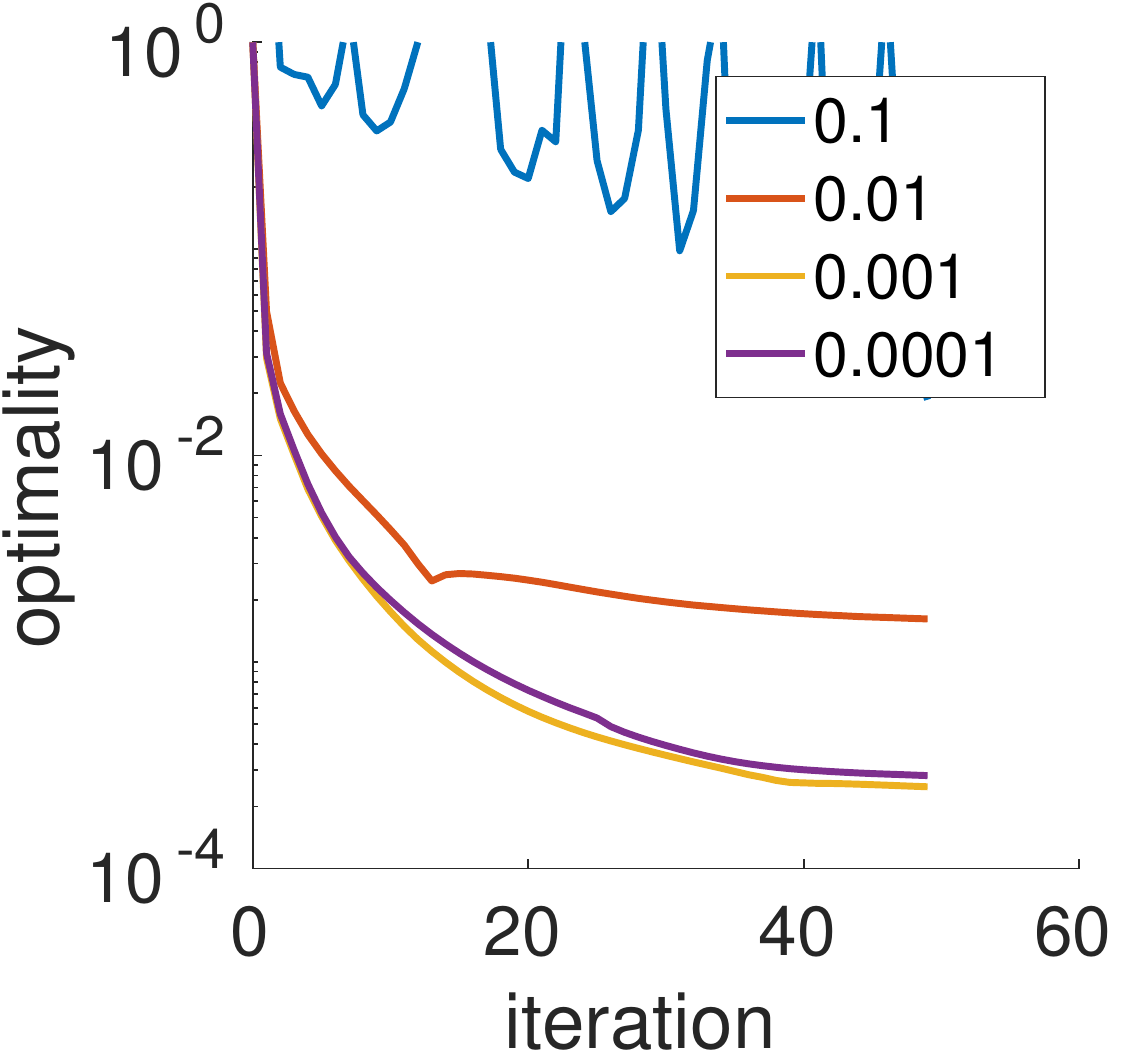}&
\includegraphics[scale=.3]{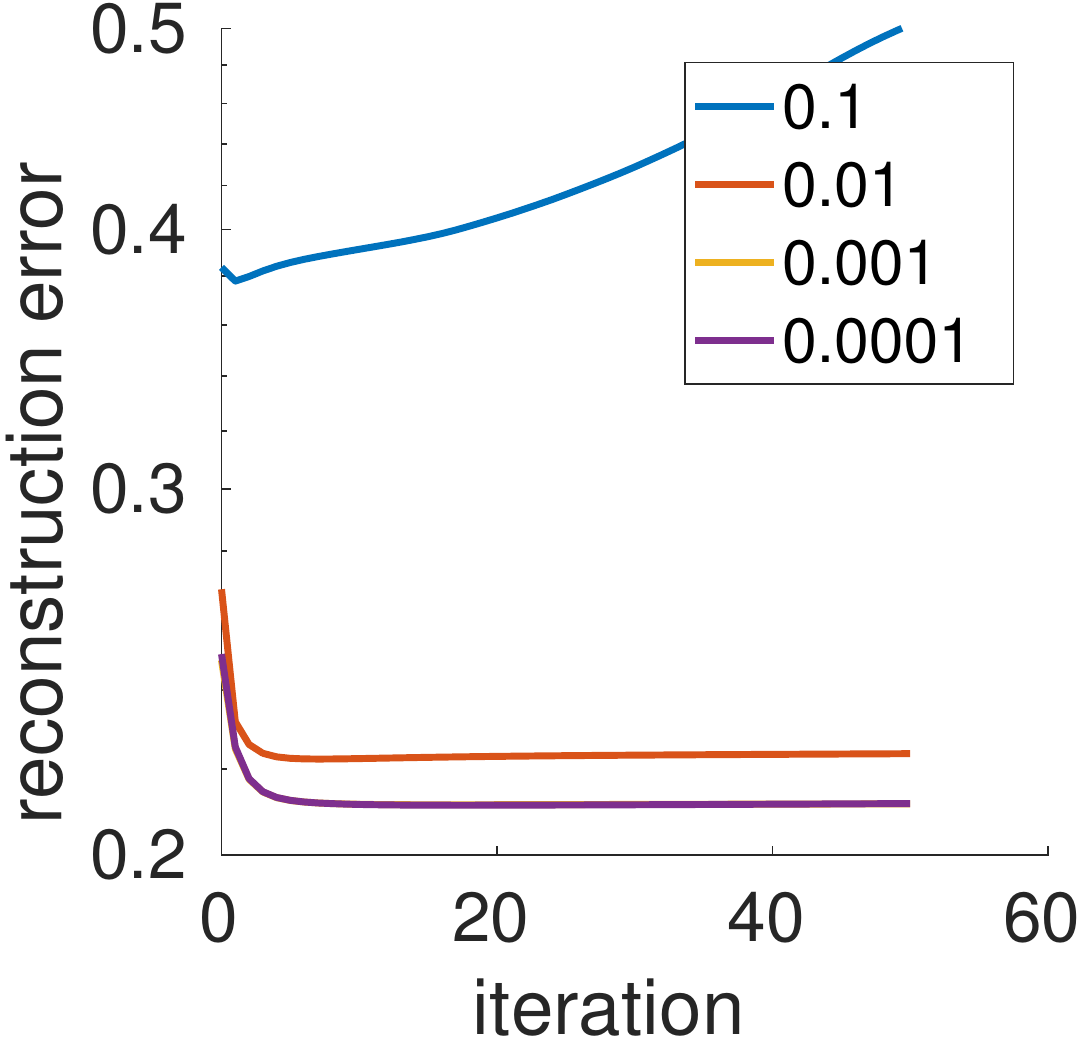}&
\includegraphics[scale=.3]{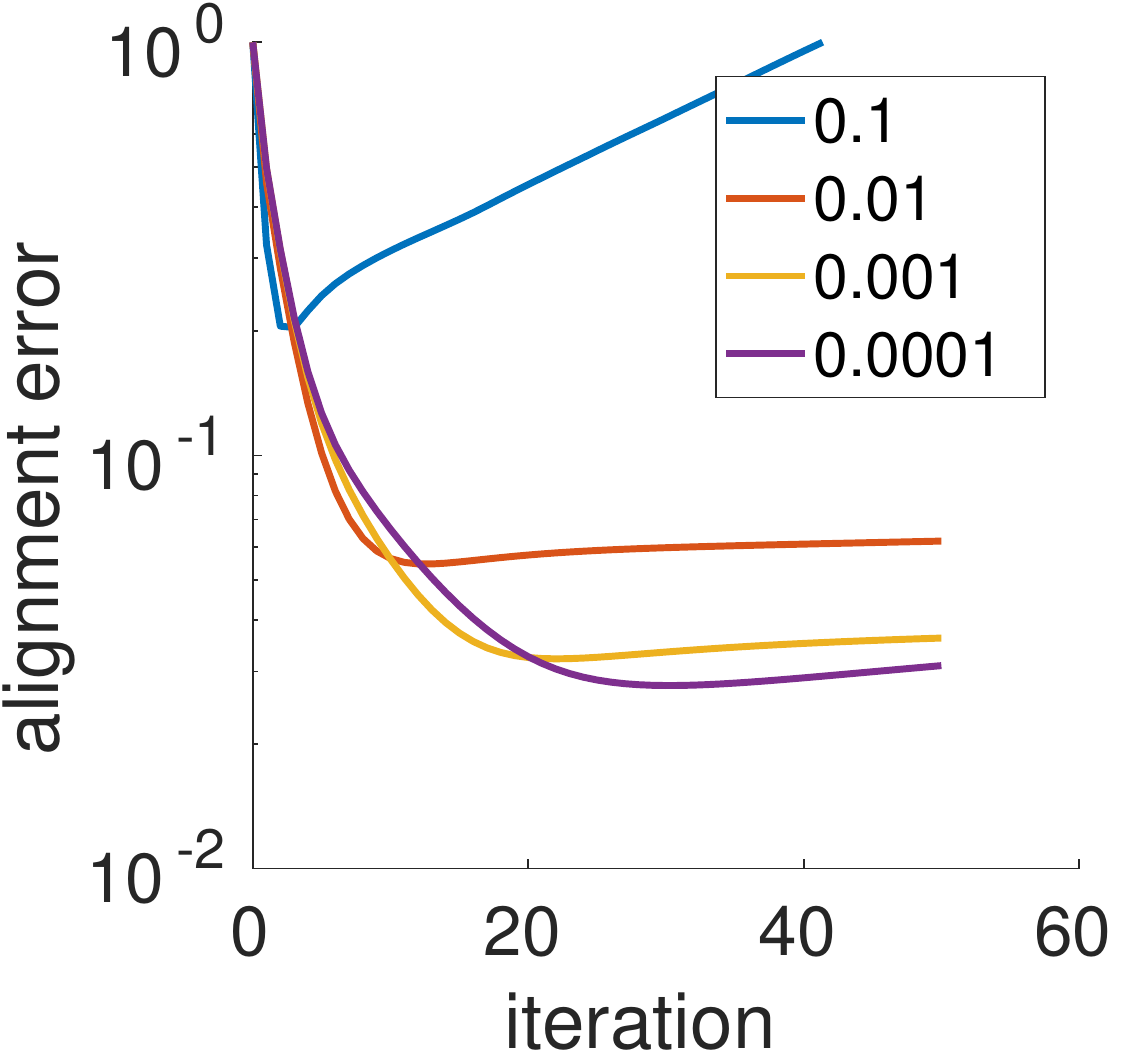}\\
&\small{$N_{\text{align}} = 10$}&\\
\end{tabular}
\caption{Convergence in terms of optimality, reconstruction error and alignment error for various reconstruction-tolerances $\epsilon$. Different rows represent varying numbers of outer versus inner alginment iterations. The top row corresponds to Algorithm \ref{algorithm:smooth}; the remaining rows correspond to Algorithm \ref{algorithm:alternating} with an increasing number of inner alignment iterations. We see that performing more than one alignment step per outer iteration does not necessarily lead to a better reconstruction and alignment.}
\label{fig:convergences}
\end{figure}

The final aspect we wish to investigate in these experiments is the sensitivity of the iterative approach to the magnitude of the alignment errors. To this end, we use the exact same alignment errors as in the previous experiments, except that we increase their magnitude by a factor of 2, 4, 8 and 16 respectively. For the alignment, we run Algorithm \ref{algorithm:smooth} with $\epsilon = 10^{-4}$ for a total of 20 outer iterations. The regularization parameter $\alpha$ is chosen using a heuristic described in section \ref{SS:tomoReconMethod}. This heuristic is based on an estimate of the misalignment and effectively assigns a higher $\alpha$ (ranging from roughly $10^1$ to $10^4$ in our experiments) for larger misalignment. In our experience, this aspect is quite crucial since a smoother reconstruction, as induced by a larger $\alpha$, improves the initial convergence of the alignment. On the contrary, choosing $\alpha$ too low may cause the alignment to stall prematurely. In practice, one could of course use a continuation, starting from large $\alpha$ and reducing it as the alignment improves.
The resulting final reconstructions (computed using $\alpha = 10^{2}$ fixed) are shown in figure \ref{fig:magnitude1} and the corresponding alignments are shown in figure  \ref{fig:magnitude2}. Even for the most severe initial misalignment, we observe that all alignment parameters are accurately recovered, except for the tomographic angle. If the misalignment exceeds the angular sampling, it becomes exceedingly hard to find the correct projection angles as nearby projections are similar. As long as enough of the projections have been aligned
correctly, the effect on the final reconstruction is not dramatic. Overall, the experiment shows that the proposed alignment procedure deteriorates gracefully as the initial misalignment increases.

\begin{figure}
\centering
\begin{tabular}{cc}
\includegraphics[scale=.6]{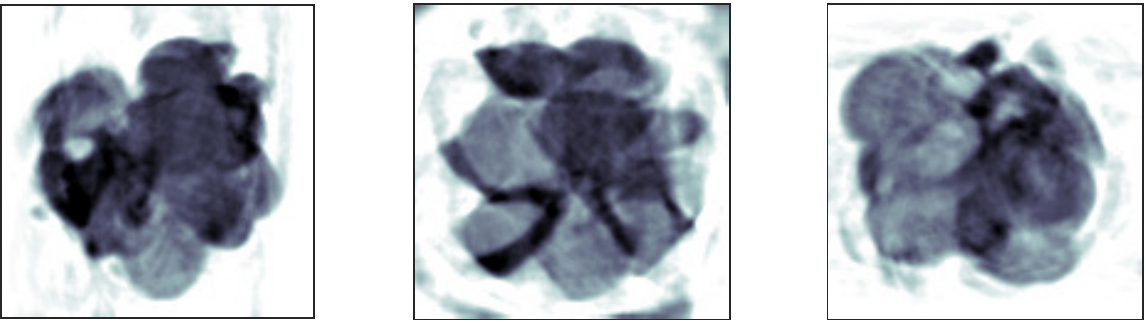}&\includegraphics[scale=.6]{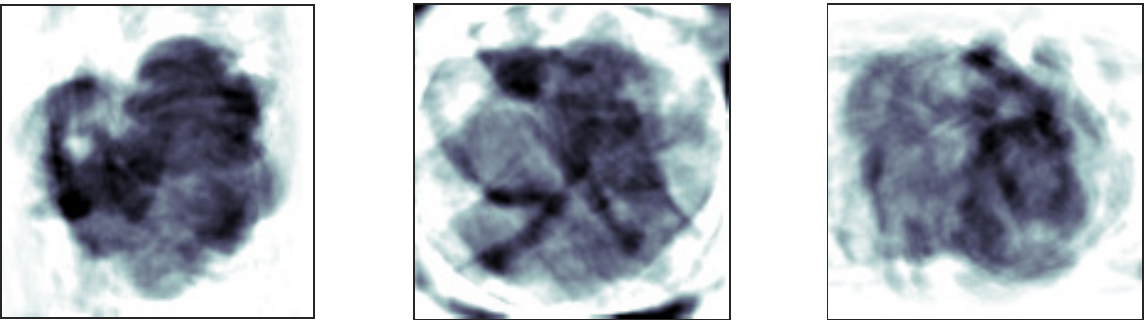}\\
\includegraphics[scale=.6]{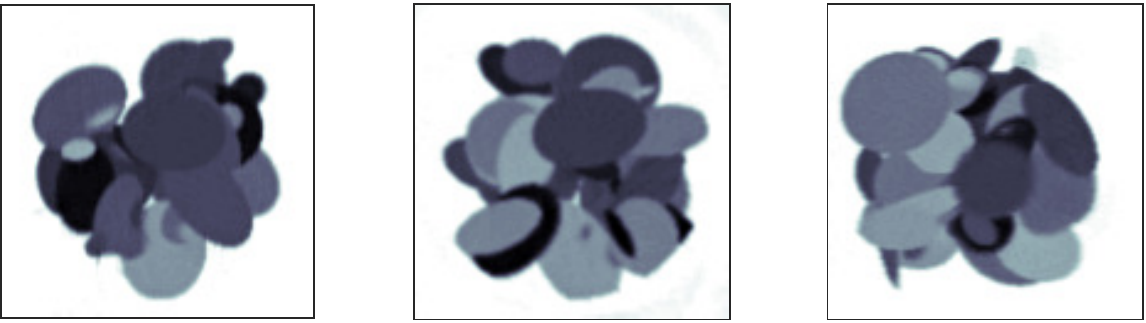}&\includegraphics[scale=.6]{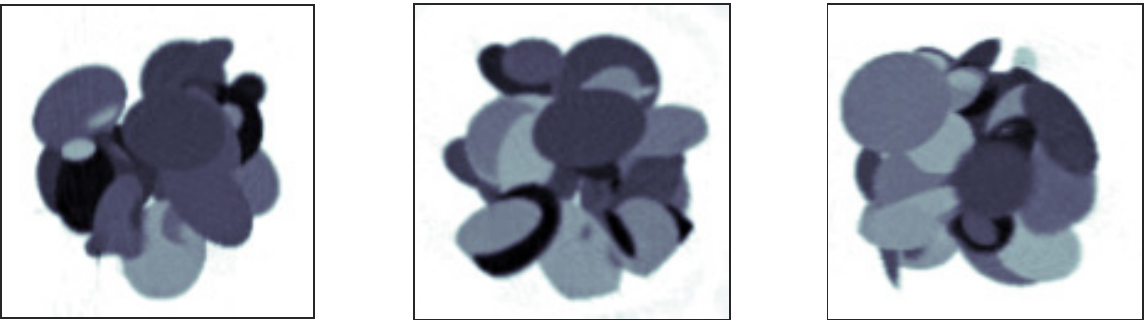}\\
\small{(initial misalignment increased by factor 2)}&\small{(initial misalignment increased by factor 4)}\\
&\\
\includegraphics[scale=.6]{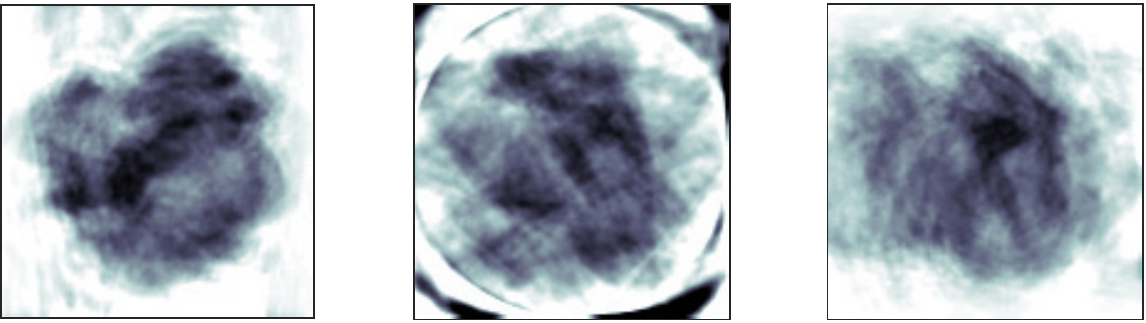}&\includegraphics[scale=.6]{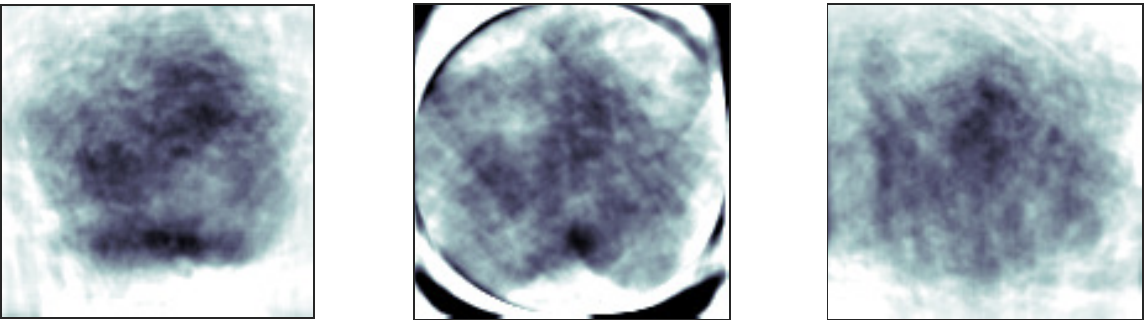}\\
\includegraphics[scale=.6]{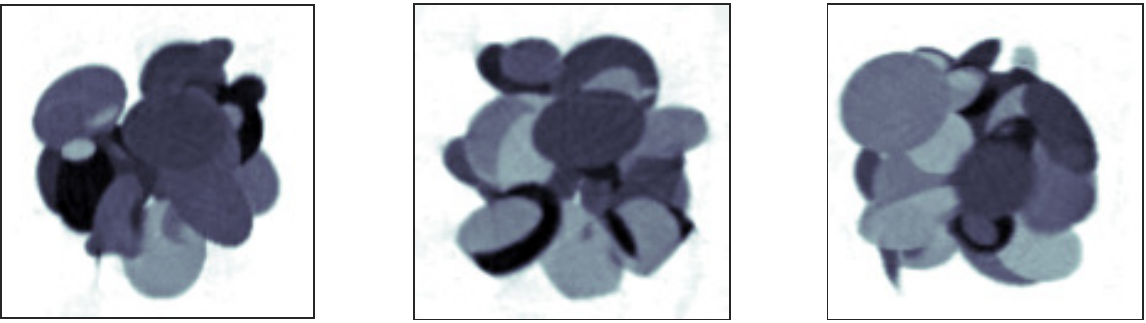}&\includegraphics[scale=.6]{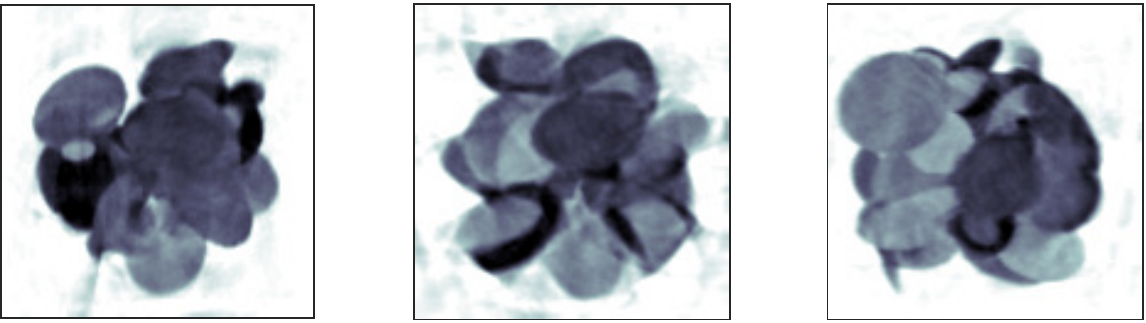}\\
\small{(initial misalignment increased by factor 8)}&\small{(initial misalignment increased by factor 16)}\\
\end{tabular}

\caption{Un-aligned (upper rows) and aligned reconstructions (lower rows) for increasing magnitude of the misalignment.}
\label{fig:magnitude1}
\end{figure}

\begin{figure}
\centering
\begin{tabular}{cc}
\includegraphics[scale=.5]{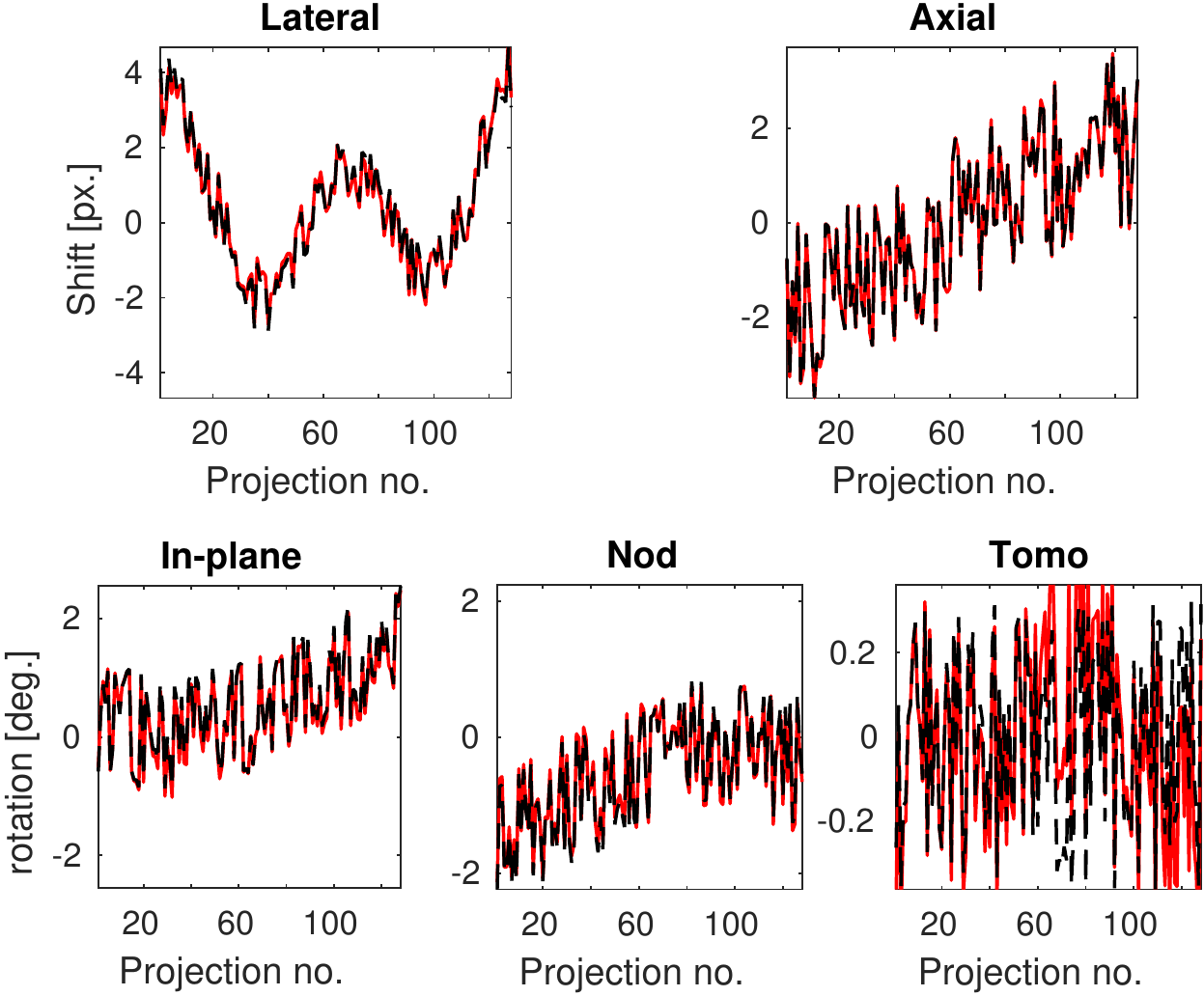}&
\includegraphics[scale=.5]{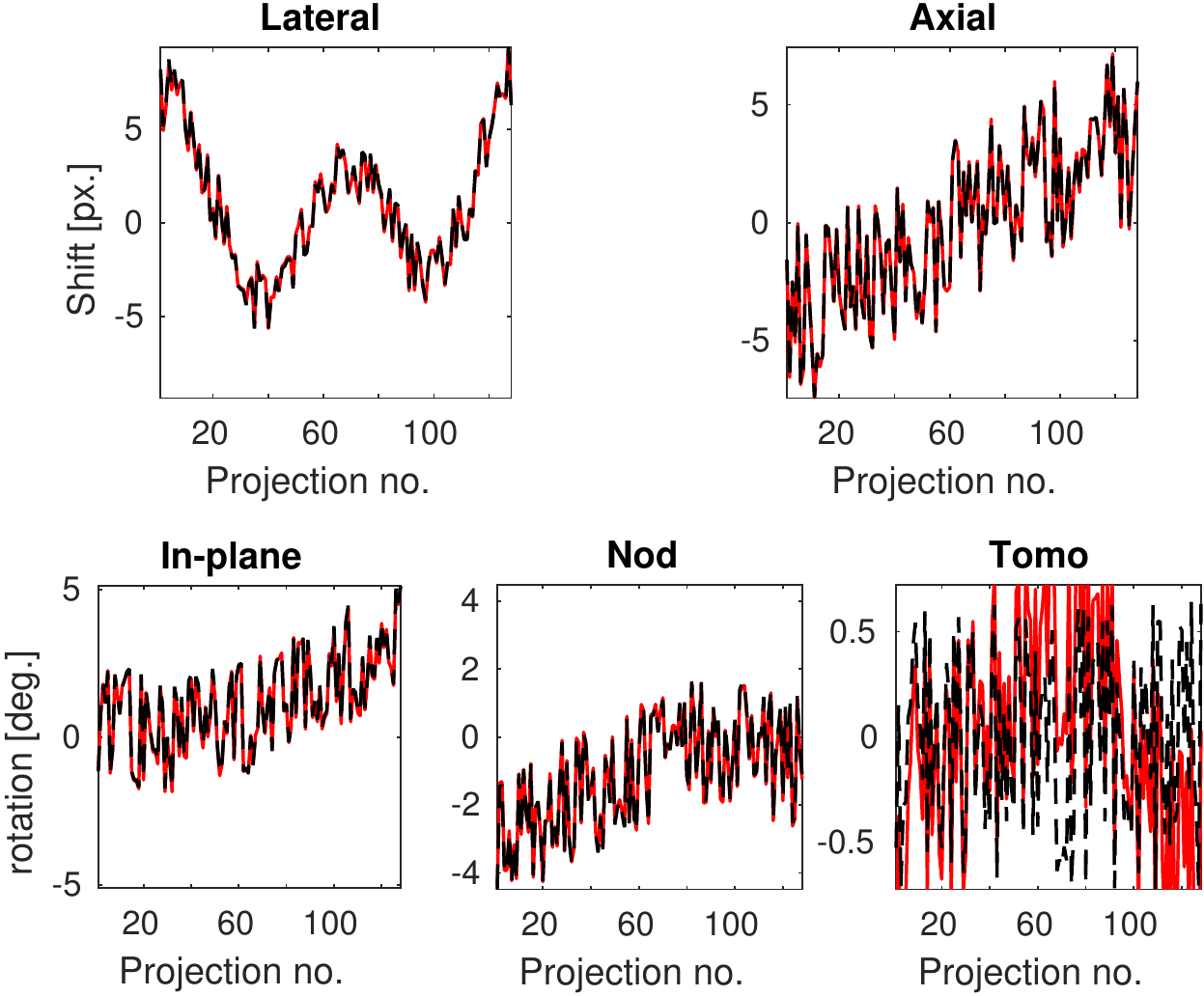}\\
\small{(factor 2)}&
\small{(factor 4)}\\
\includegraphics[scale=.5]{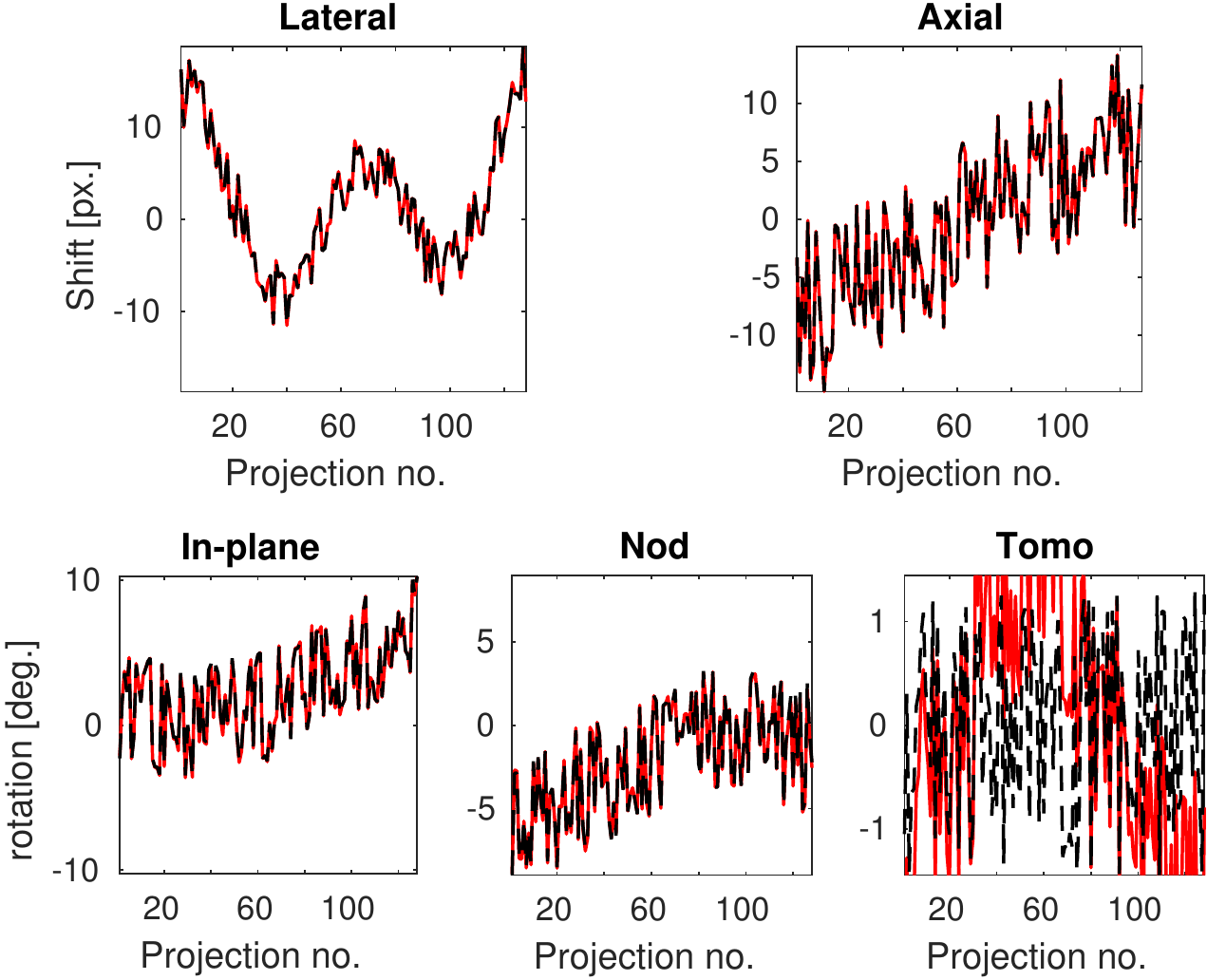}&
\includegraphics[scale=.5]{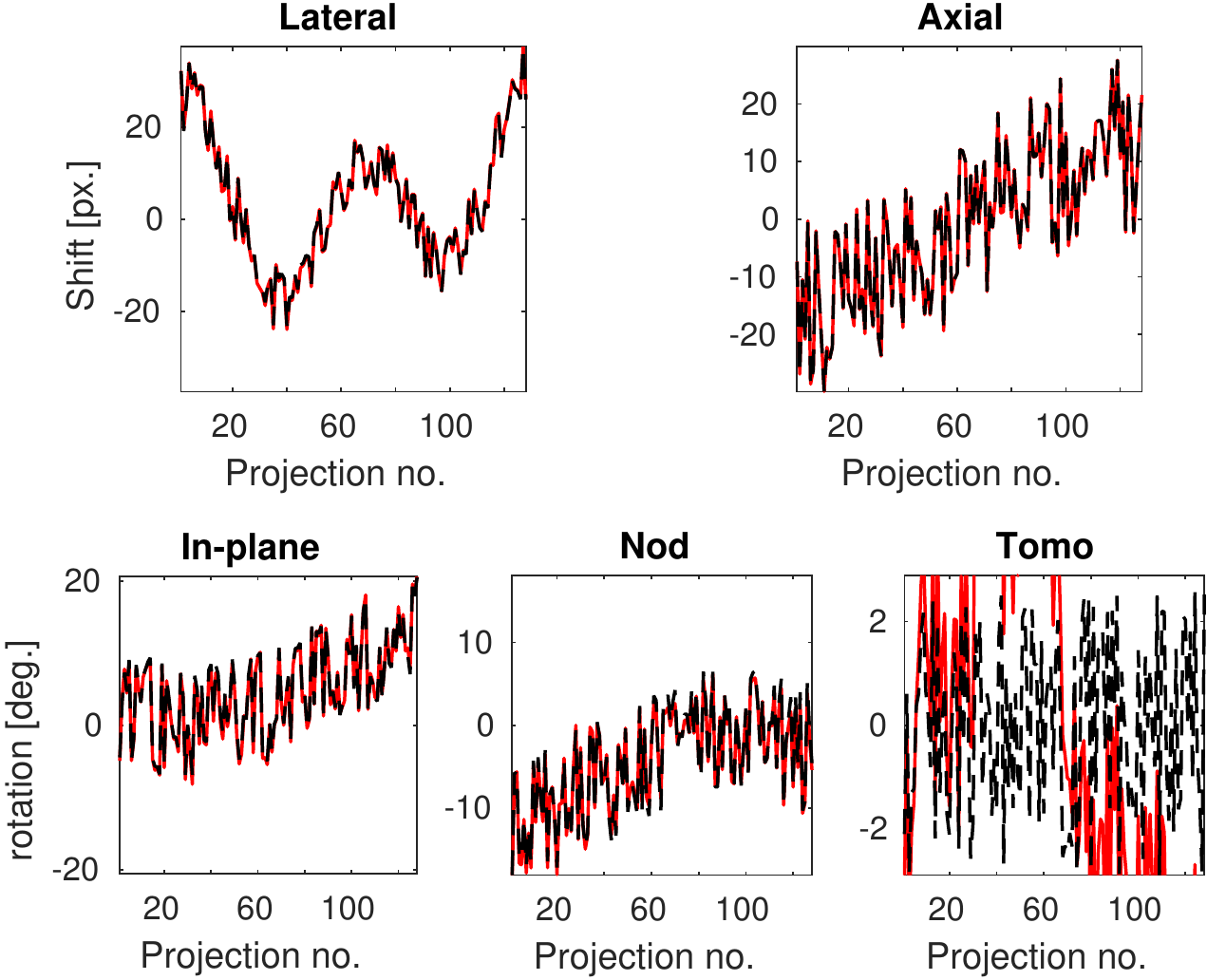}\\
\small{(factor 8)}&
\small{(factor 16)}\\
\end{tabular}
\caption{Retreived alignment parameters for increasing magnitude of the misalignment.}
\label{fig:magnitude2}
\end{figure}

% Region-of-interest example
\subsection{Alignment of truncated projection data} \label{SS:Results-RoI}

So far, we have only considered settings where the object is fully captured in all of the tomographic projections. In many applications, however, only a limited \emph{region-of-interest} (RoI) of an extended sample is imaged due to setup limitations or practical considerations. This setting can be modelled by composing the projection matrix $W(\mbf a)$ with an additional truncation matrix $T_{\textup{roi}} \in \mR^{m_{\textup{roi}} \times m}$ that restricts the (hypothetical) full-field tomographic data to the information captured by the $m_{\textup{roi}}$ detector pixels:
\begin{equation}
 W_{\textup{roi}}(\mbf a) := T _{\textup{roi}} W(\mbf a).
\end{equation}
The limited field-of-view allows for the integrated object mass in the detection domain to change while the sample is rotated. Most importantly, this effect rules out the theoretical exactness of center-of-mass-based alignment methods, although approximate schemes, where smooth window functions dampen the in- and outflow of object mass, have been proposed \cite{sanders_physically_2015}. On the contrary, our alignment algorithms remain applicable in the RoI-setting. We simply have to exchange the projector $ W(\mbf a)$ by $W_{\textup{roi}}(\mbf a)$ and otherwise follow exactly the same theoretical and numerical approach as in the full-field case. Note that the gradient penalty in the Tikhonov-regularized reconstruction (see \ref{SS:tomoReconMethod}) results in a smooth extrapolation of the recovered object outside the region-of-interest so that padding of the truncated data (see e.g. \cite{KyrieleisEtAl2011_RoIPractLimits}) is not necessary.

As a test case, we consider a $256^3$-sized object composed of randomly generated ellipsoids in a cylindrical support domain and simulate 128 projections in between $0^{\circ}$ and $160^\circ$, truncated to the central $128^2$ pixels. The tomographic data is simulated under both systematic and random lateral- and axial shifts of the object, compare figure~\ref{fig:alignment}. Rotational misalignment is omitted in this test case. The exact phantom (central slice), the simulated RoI-tomographic data (central sinogram) and the misalignment shifts (dotted lines) are shown in figures \ref{fig:RoI}(a), (b) and (c), respectively. We apply Algorithm~\ref{algorithm:smooth} to the data, where the iterations are automatically stopped after 39 alternating reconstruction- and alignment steps when the computed correction in the alignment parameters $\mbf a ^{(k)} - \mbf a ^{(k-1)}$ drops below $0.05$ pixels in maximum norm.

\begin{figure}
 \centering
 \includegraphics[width=\textwidth]{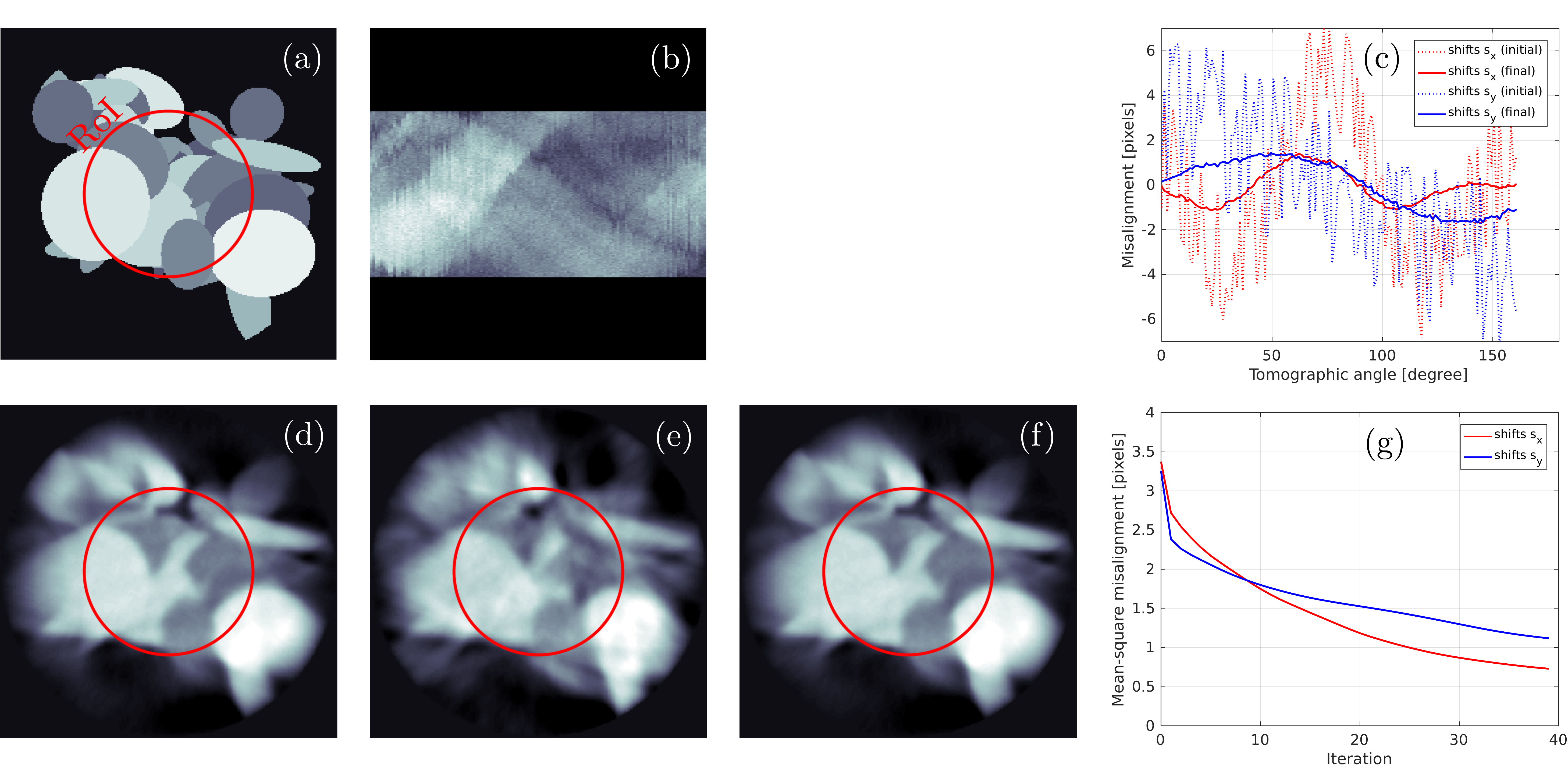}
 \caption{Alignment of truncated projection data by Algorithm~\ref{algorithm:smooth}: \newline 
 (a) Central slice (perpendicular to the tomographic axis) of the exact phantom used in the simulation. \newline
 (b) Central sinogram of the simulated truncated tomographic projections. \newline
 %(c) Same sinogram, corrected using the reconstructed alignment parameters. \newline
 (c) Misalignment shifts of the object perpendicular (blue) and along (red) the tomographic axis. Dotted lines: initial misalignment. Solid: final residual misalignment after application of the algorithm. \newline
 (d), (e), (f) Same slice as in (a) of reconstructed objects with exact, initial and final alignment parameters, respectively (all other reconstruction parameters being equal). The red circles mark the region-of-interest that is fully contained in the truncated projections (up to in- and out-movement due to misalignment). \newline
 (g) Convergence plot: decrease of the root-mean-square difference between estimated and true alignment shifts.}
 \label{fig:RoI}
\end{figure}

The central slice of the final object reconstruction is shown in figure \ref{fig:RoI}(f), the difference between the reconstructed- and exact alignment parameters is plotted in (c) (solid lines). While rapid variations of shifts $\{(s_{i,x}, s_{i,y})\}$ with the tomographic angles $\theta_{i,xz}$ are almost completely eliminated by the alignment method, some slowly varying misalignment components remain. This is too not surprising since region-of-interest tomography is non-unique even without an additional alignment problem to be solved \cite{natterer_mathematics_2001}. Notably, however, figure \ref{fig:RoI}(c) reveals that also the bulk components of the simulated object motion are significantly damped compared to the initial misalignment. Likewise, the motion artifacts in the recovered object are almost invisible in the final reconstruction whereas quite pronounced initially, compare figures \ref{fig:RoI}(e) and (f). Visually, the achieved reconstruction using the fitted alignment parameters is indeed of comparable quality as a reconstruction using the \emph{exact} alignment, which is shown in figure \ref{fig:RoI}(d) for comparison. The convergence plot in figure \ref{fig:RoI}(g) reveals the fit of the alignment parameters to improve only very slowly, which indicates that the alignment problem is highly ill-conditioned in this setting. Indeed, the curves suggest that further improvements might be achievable by performing significantly more iterations.
%Clearly however, the considered example demonstrates that the proposed alignment method may considerably improve reconstruction quality also in region-of-interest tomography.

% Electron tomography example
\subsection{Application to an electron tomography data set} \label{SS:ETomo}

For evaluation of our alignment scheme in a real-world context, we apply it to an electron tomography data set (HAADF-STEM). The sample is given by an assembly of gold nanoparticles of $20$ nanometers mean diameter, embedded in a polymeric matrix. See \cite{VanAarle2015_ASTRA_ETomoSpheres} and references therein for experimental details and first tomographic reconstructions. Although actually measured as a dual-axis tilt series, i.e.\ for two perpendicular tomographic axes as sketched in figure \ref{fig:acquisition}(c), we consider only a single-axis sub-data set for simplicity. After background substraction and cropping to the relevant part of the field-of-view, the tomographic data consists of 75 projections of size $512 \times 512$, uniformly covering an angular range of $150^\circ$. As our alternating minimization alignment method is not global, we prealign the data by computing the center-of-mass (COM) of every projection and shift these to center the COM. The central sinogram of the prealigned tomographic is shown in figure \ref{fig:ETomo}(d).

In order to limit the computational effort, we bin the data to size $256 \times 256 \times 75$ and reconstruct the object in a cylindrical support domain enclosed by a $256^3$ voxel cube. For each projection, lateral and axial shifts $s_{i,x}, s_{i,y}$ as well as in-plane- and pitch-rotation angles $\theta_{i,xy}, \theta_{i, yz}$ are fitted (compare figure~\ref{fig:alignment}), i.e.\ a fully 3D alignment is performed except for perturbations in the tomographic angles. The regularization parameter $\alpha$ in the Tikhonov cost functional is chosen according to an estimated misalignment of 10 pixels, see \ref{SS:tomoReconMethod}.
% and \ref{SS:representAlign}

Slices of an object recovered from the prealigned data (corresponding to the initial iterate of the alignment algorithm) are plotted in figures \ref{fig:ETomo}(a)-(c). Figures \ref{fig:ETomo}(e)-(g) show the same slices for a reconstruction based on the final alignment parameters obtained after 47 alternating reconstruction- and alignment steps of Algorithm \ref{algorithm:smooth}. Note that the recovered objects are not the (smooth) intermediate reconstructions computed by the alignment method. Rather, they are calculated using a regularization parameter $\alpha$ that is reduced by a factor of 100 in order to facilitate the identification of misalignment artifacts and the comparison to the results in \cite{VanAarle2015_ASTRA_ETomoSpheres}. The fitted alignment parameters are plotted in  figure \ref{fig:ETomo}(h).

\begin{figure}
 \centering
 \includegraphics[width=\textwidth]{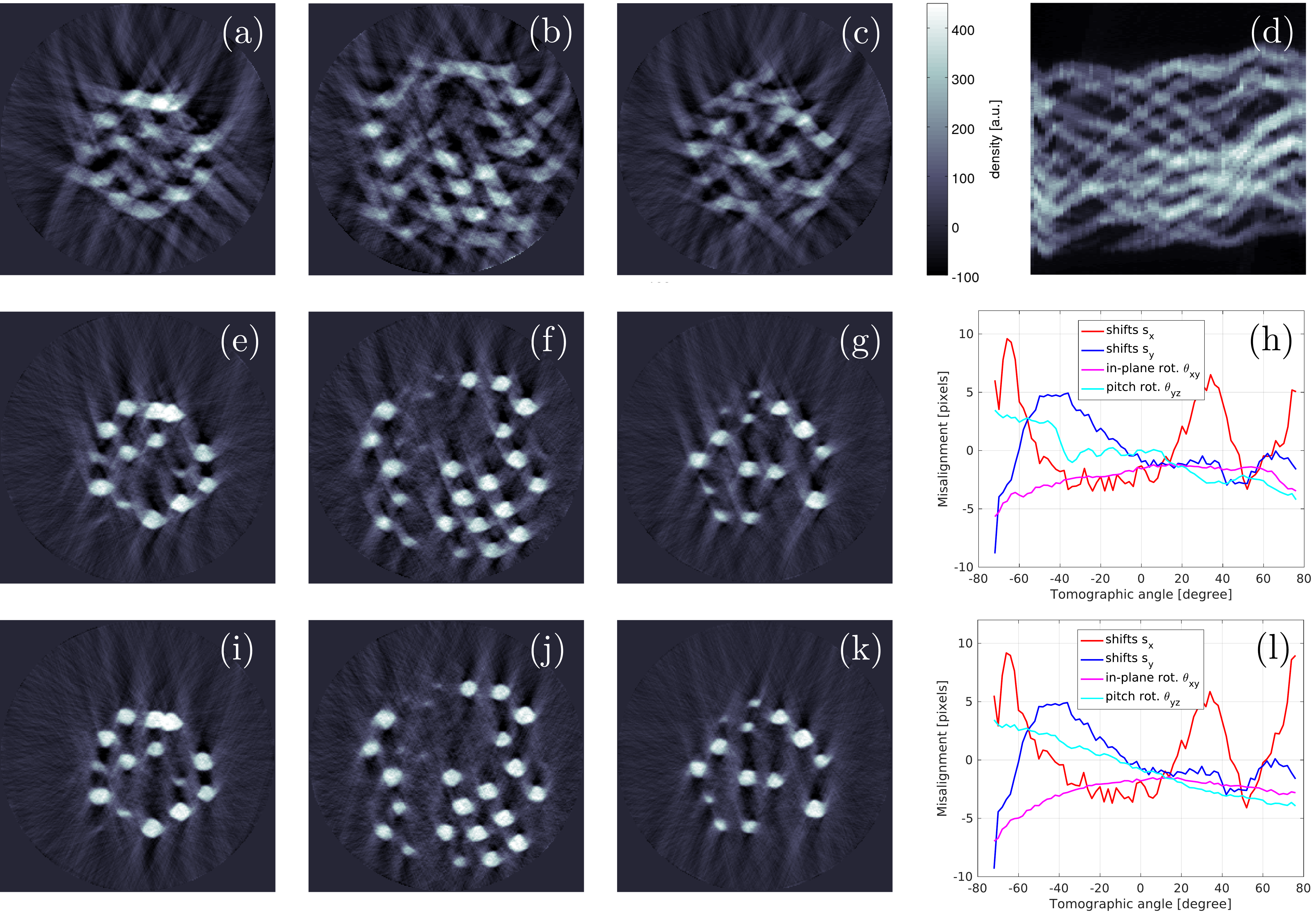}
 \caption{Alignment of an electron tomography data set (HAADF-STEM) \cite{VanAarle2015_ASTRA_ETomoSpheres}: \newline (a)-(c) Slices of a $256^3$-sized object reconstruction after center-of-mass prealignment  (perpendicular to the tomographic axis at heights $y = 64$, $128$ and $192$ voxels). \newline
 (d) Central sinogram of the prealigned tomographic data. \newline
 (e)-(g) Same reconstruction as in (a)-(c) but computed using the final alignment parameters $\mathbf{a}^{(k_{\textup{stop}})}$ obtained by applying Algorithm \ref{algorithm:smooth} to the prealigned data. \newline
 (h) Fitted alignment parameters $\mathbf{a}^{(k_{\textup{stop}})}$. Rotation angles are scaled according to the corresponding mean motions of the voxels in the cylindrical supporting domain. \newline
  (i)-(l) Same plots as (e)-(h) for results from a second run of Algorithm \ref{algorithm:smooth} using a Kaczmarz-type tomographic reconstruction method with non-negativity constraint. \newline All plotted objects have been reconstructed with the same Tikhonov regularization.}
 \label{fig:ETomo}
\end{figure}

Figure \ref{fig:ETomo} reveals that the stripe-shaped misalignment artifacts are considerably reduced by the alternating-minimization alignment compared to the initial prealignment: figures \ref{fig:ETomo}(e)-(g) show well-defined compact particles as anticipated for the considered sample. However, in addition to shadow- and stripe artifacts associated with the incomplete tomographic angular range (\emph{missing wedge} of $30^\circ$) and the coarse angular sampling, respectively, the plots show some residual crescent-shaped appendices of the particles along specific directions. This indicates a residual misalignment in the determined position of the tomographic axis.

% Remove bad paragraph-pagebreak behavior here:
\clubpenalty = 10000

The systematic occurence of negative values in the recovered objects suggests that the alignment might be improved by imposing a non-negativity constraint in the tomographic reconstruction. As this is not possible for the conjugate-gradients-based Tikhonov regularization, we run Algorithm \ref{algorithm:smooth} a second time using the alternative Kaczmarz-type tomographic reconstruction method described in \ref{SS:tomoReconMethod}. The regularization parameter $\alpha$ and all other settings are chosen exactly as before.

Notably, convergence occurs much faster in this case, the stopping criterion of less than $0.05$ pixel increment in the alignment parameters being met already after 15 iterations. The final fit of alignment parameters, plotted in figure \ref{fig:ETomo}(l), deviates from the previous results (figure \ref{fig:ETomo}(h)) in the order of $\pm 1$ pixels. In order to compare the quality of the two fits, we compute a final (Tikhonov-)reconstruction  with the same regularization as for the results shown in figures \ref{fig:ETomo}(e)-(g), yet using the alignment parameters obtained from the Kaczmacz-type approach. The recovered object slices are shown in figures \ref{fig:ETomo}(i)-(k). While the artifacts associated with a missing-wedge- and the angular sampling remain, the directed crescent-structures are almost completely eliminated. This indicates a significantly improved alignment of the tomographic axis.

%along with the central slice of

\section{Conclusions and discussion}
\label{conclusions}

In many applications of tomographic imaging, the experimental setup deviates geometrically from the idealized mathematical model. Such a misalignment can cause severe distortions in the reconstructed image. Therefore, an alignment procedure is needed to match the mathematical model to the true tomographic geometry. In some cases, this alignment can be done prior to reconstruction by simple transformations of the projection images. For general three-dimensional misaligment of the setup, reconstruction and alignment have to be done jointly.
In this paper, we discuss the alignment problem from an algebraic point-of-view and propose an algorithmic framework for solving joint reconstruction and alignment problems. By posing the joint problem as a separable least-squares problem, we can exploit the special structure of such problems: they are quadratic in the image but generally non-convex in the alignment parameters. Since these optimization problems allow for a unique reconstructed image for a given set alignment parameters, we can eliminate the image by solving a (regularized) least-squares problem. The remaining problem is shown to be smooth in the alignment parameters (given a sufficiently differentiable interpolation scheme), even when using non-smooth regularization for the image reconstruction.
The alignment problem may include additional regularization terms on the alignment parameters and can be solved using a gradient-based method. We present a general mathematical framework for designing and analyzing such gradient-based methods and discuss their convergence properties. Under the additional assumption of (local) convexity, we show that a linear convergence rate can be achieved even when performing inexact reconstructions.

We present numerical experiments on both phantom and real data to illustrate the proposed approach. For simulated full-field tomographic data, we find that small to moderate perturbations in almost all alignment parameters can be stably retrieved by Algorithm~\ref{algorithm:smooth}, fitting the alignment parameters by simple gradient-descent steps. The only exception is given by inaccuracies in the tomographic tilt angle, for which we observe a slight instability of the reconstruction to low-frequency errors. 
By its generic nature, our method can also be applied in the case of truncated projection data as measured in region-of-interest tomography.
Indeed, our simulations show that it may correctly detect misalignment in the data even in this setting, for which already tomographic reconstruction without an alignment problem is non-unique. However, convergence turns out to be rather slow in this case. Aligning an experimental electron tomography data set by our algorithm, we demonstrate its practical applicability in more challenging, real situations, in particular involving a limited range of tomographic angles. We furthermore find that the alignment can be improved by
% applying an alternative Kaczmarz-type tomographic reconstruction method
additionally imposing a non-negativity constraint on the reconstructed images. This reveals that the efficiency of the presented alignment framework greatly benefits from the exploitation of available a priori knowledge on the imaged objects. Optimizing the incorporation of constraints is thus an important goal for further development.

It should be emphasized that, rather than developing a highly-taylored alignment method for specific types of datasets, we set out to describe a general algorithmic framework in which joint reconstruction-alignment methods can be formulated and analyzed. Our implementation contains certain features to make it generic and robust. This includes a linesearch procedure, estimation of the regularization parameter, and the ability to work with a wide range of acquisition setups. Alignment of challenging large-scale datasets, however, requires further tuning and a more efficient implementation, which is subject to future research.

\section*{Acknowledgements}
This work was funded in part by Netherlands Organisation for Scientific Research (NWO) through projects 613.009.032 and 639.073.506 and by the Deutsche Forschungsgemeinschaft (DFG) through project C02 of SFB 755 - Nanoscale Photonic Imaging.

\clearpage
\bibliographystyle{abbrvIP}
\bibliography{mybib}

\clearpage
\appendix

\section{Implementation} \label{S:implementation}

\subsection{Discretization of the X-ray transform} \label{S:DiscXray}

We express the X-ray transform of a three-dimensional object $u$ as
\[
p(\mathbf{s},\boldsymbol{\eta}) = \int_{\mathbb{R}}\!\mathrm{d}t\, u(\mathbf{s} + t\boldsymbol{\eta}),
\]
where $\mbf{s}\in\mathbb{R}^3$ and $\boldsymbol{\eta}\in\mathbb{S}^{2}$, $\mathbb{S}^{2} = \{\mbf{x}\in\mathbb{R}^3\,\, |\,\, \|\mbf{x}\|_2 = 1\}$ denoting unit sphere, determine the origin and direction of the X-ray. The projection data consists of a collection of such projections $p_i = p(\mathbf{s}_i,\boldsymbol{\eta}_i)$, for $i = 1, \ldots, m$. Sampling the object $u$ on a grid of $n$ voxels with values $u_i$ for $i = 1, \ldots, n$ we discretize the X-ray transform as
\[
\mathbf{p} = W(\mathbf{a}) \mbf u.
\]
The exact form of the system matrix $W(\mbf a ) \in \mR^{m \times n}$ depends on the applied interpolation method, see \ref{SS:interpolation}.
The tomographic data can be grouped into the different projections $\mbf p = (\mbf p_1, \mbf p_2, \ldots, \mbf p_{N_{\textup{proj}}})^{\textup T}$, each acquired at a specific, yet possibly unknown set of alignment parameters $\mathbf{a}_i$ - a vector parametrizing the six degrees of freedom of a rigid object motion.
% The tomographic projections can be grouped naturally in such a way that each group can be described by one set of alignment parameters $\mathbf{a}_i$.
The projection matrix then has a block-structure
\[
W(\mathbf{a}) = \left(\begin{matrix}
W_{1}(\mathbf{a}_1) \\
W_{2}(\mathbf{a}_2) \\
\vdots\\
W_{N_{\textup{proj}}}(\mathbf{a}_{N_{\textup{proj}}})
\end{matrix}\right).
\]
Importantly, this also induces a block-structure of the Jacobian to be evaluated in Algorithm~\ref{algorithm:smooth}:
\[
G(\mathbf{a},\mathbf{u}) =  \frac{\partial W(\mbf a) \mbf u }{\partial  \mbf a} = \left(\begin{matrix}
\frac{\partial W_{1}}{\partial \mathbf{a}_1}(\mathbf{a}_1) \mbf u  & & &0 \\
& \frac{\partial W_{2}}{\partial \mathbf{a}_2}(\mathbf{a}_2) \mbf u  & & \\
&& \ddots \\
0 &&& \frac{\partial W_{N_{\textup{proj}}}}{\partial \mathbf{a}_{N_{\textup{proj}}}}(\mathbf{a}_{N_{\textup{proj}}}) \mbf u
\end{matrix}\right).
\]

\subsection{Parametrization of the alignment transforms} \label{SS:representAlign}

We represent the alignment parameters for one projection as $\mathbf{a}_i = ( \mathbf{s}_i, \boldsymbol{\theta}_i ) $, where the shift-vector $\mathbf{s}_i = (s_{i,x}, s_{i,y}, s_{i,z}) \in \mR^3$ contains scalar translations along the different dimensions and $\boldsymbol{\theta}_i = (\theta_{i,xy}, \theta_{i,yz}, \theta_{i,zx}) \in [-\pi; \pi)$ contains the rotation angles to be applied in different coordinate planes, see figure~\ref{fig:alignment}.

The projection matrix $W_{i}$, that models the parallel-beam tomographic projection of the object under the rigid geometrical transform $\mathbf{a}_i$, can be written as
\begin{equation}
 W_{i}(\mathbf{a}_i) = S_z R(\mathbf a_i), \label{eq:ProjRotateAndSum}
\end{equation}
where $R(\mathbf a_i)$ is a matrix implementing rotations and shifts and $S_z$ sums the voxel values of three-dimensional arrays along the $z$-dimension. Note that only $S_z$ and not $R$ is specific to the considered parallel-beam geometry. Hence, it would be easy to implement alignment problems also for cone-beam setups, for example, by simply exchanging $S_z$ with the cone-beam projector for a fixed source position.

As shifts and rotations do not commute, we have to follow some convention on the order in which the different transforms are applied. We choose the following:
\begin{enumerate}
 \item Rotation in the $z$-$x$-plane (\emph{tomographic rotation})
 \item Rotation in the $y$-$z$-plane (\emph{pitch rotation})
 \item Shifts in along the $x, y, z$-dimensions
 \item Rotation in the $x$-$y$-plane (\emph{in-plane rotation})
\end{enumerate}
As the naming suggests it, we take the $y$-axis as the tomographic axis and apply the desired rotation about this axis first before modeling additional (undesired) motions. The benefit of this ordering is that \emph{trivial} components of the alignment transforms $\mathbf{a}$ can be easily identified and removed. These correspond to constant transforms of the entire object over the whole tomographic acquisition (see e.g.\ \cite{basu_uniqueness_2000}):
\begin{itemize}
 \item $\{\theta_{i,zx}\}$: The mean corresponds to a constant rotation about the tomographic axis
%  about the tomographic axis
 \item $\{\theta_{i,xy}\}, \{\theta_{i,yz}\}$: Joint sinusoidal variations $(\theta_{i,xy}, \theta_{i,yz}) = a ( \sin( \theta_{i,zx} + b ), \cos( \theta_{i,zx} + b ) )$ for $a > 0$, $b \in \mR$ correspond to a uniform tilt of the tomographic axis
 \item $\{s_{i,x} \}$: Sinusoidal components $s_{i,x} = c \sin( \theta_{i,zx} + d)$ correspond to a constant translation perpendicular to the tomographic axis
 \item $\{s_{i,y} \}$: The mean corresponds to a constant translation along the tomographic axis
 \item $\{s_{i,z} \}$: Shifts along the projection dimension; redundant in parallel-beam geometry
\end{itemize}

In order to ensure unambiguous results, we remove the above components from any simulated and reconstructed alignment transforms $\mathbf{a}$ by substracting the orthogonal projections onto the redundant modes.

\subsection{Computation of rigid transforms and interpolation} \label{SS:interpolation}

We consider discrete objects $\mbf{u} \in \mR^{N_x \times N_y \times N_z}$, that sample the continuous density $u:\mR^3 \to \mR$ on equidistant grid points $\mathbf x_{\textup{s}} = (x_{\textup{s}}, y_{\textup{s}}, z_{\textup{s}} ) \in \{1,\ldots,N_x\} \times \{1,\ldots,N_y\} \times \{1,\ldots,N_z\} =: D_{\mbf{u}}$. The action of the geometrical transform on $\mbf{u}$ can be written as
\begin{equation}
  R(\mathbf a_i)(u)_{\mathbf x_{\textup{s}}  } = \tilde u(A_R(\mathbf a_i)^{-1}( \mathbf x_{\textup{s}}  ) ) \label{eq:TransformOnCoords}
\end{equation}
for all $\mathbf x_{\textup{s}}  \in D_{\mbf{u}}$. Here, $A_R(\mathbf a_i)^{-1} : \mR^3 \to \mR^3$ denotes the inverse of the affine-linear map that implements the rotation and shift parametrized by $\boldsymbol \theta_i, \mathbf s_i$ and $\tilde u$ is some interpolating function for $\mbf{u}$ that allows its approximation
% The required evaluation points $A_R(\mathbf a_i)^{-1}( \mathbf x )$ of $f$ in \eqref{eq:TransformOnCoords} will in general not coincide with any sampling point in $D_{\mbf{u}}$. Thus, the above approach requires an interpolation method to evaluate $f$ at
on arbitrary points in $\mR^3$.

We use \emph{trilinear} and \emph{tricubic} interpolators of the form
\begin{align}
 \tilde u(\mathbf x ) &= \sum_{\mathbf x_{\textup{s}}\in D_{\mbf{u}}} u_{\mathbf x_{\textup{s}}  } k_{\textup{3d}} (\mathbf x - \mathbf x_{\textup{s}}), \quad k_{\textup{3d}}(\mathbf x ) = k(x) k(y) k(z), \quad k \in \{k_{\textup{lin}}, k_{\textup{cubic}}\}  \label{eq:NLinCubicInterp} \\
 k_{\textup{lin}}(x) &:= \begin{cases}
                         1-|x| & |x| \in [0;1] \\
                         0   &\text{else}
                        \end{cases}, \quad
 k_{\textup{cubic}}(x) := \begin{cases}
                         \frac 3 2 |x|^3 - \frac 5 2 |x|^2 + 1 & |x| \in [0;1] \\
                        - \frac 1 2 |x|^3 + \frac 5 2 |x|^2 - 4 |x| + 2 & |x| \in [1;2] \\
                         0   &\text{else}
                        \end{cases} \nonumber
\end{align}
for all $\mathbf x \in \mR^3$. The interpolation kernels $k \in \{k_{\textup{lin}}, k_{\textup{cubic}}\}$ have the properties $k(0) = 1$ and $k(x) = 0$ for all $x \in \mZ \setminus \{ 0\}$. Hence, $\tilde u$ indeed interpolates $\mbf{u}$, i.e.\ $\tilde u ( \mathbf x_{\textup s} ) =  u_{\mathbf x_{\textup{s}}  }$ holds on all sampling points $\mathbf x_{\textup s} \in D_{\mbf{u}}$. Furthermore, $k_{\textup{lin}}$ is Lipschitz-continuous and $k_{\textup{cubic}}$ is continuously differentiable with Lipschitz-continuous derivative. Accordingly, the same smoothness properties hold for interpolator defined in \eqref{eq:NLinCubicInterp} and its gradient may be evaluated via
\begin{equation}
 \nabla \tilde u(\mathbf x ) = \sum_{\mathbf x_{\textup{s}}\in D_{\mbf{u}}} u_{\mathbf x_{\textup{s}}  } \nabla k_{\textup{3d}} (\mathbf x - \mathbf x_{\textup{s}})
\end{equation}

Differentiability of the interpolation method for $\mbf{u}$ allows to compute the derivative of $R(\mathbf a_i)(\mbf{u})_{\mathbf x_{\textup{s}}  }$ with respect to the alignment parameters $\mathbf a_i$:
\begin{equation}
 \frac{\partial R(\mathbf a_i)(\mbf{u})_{\mathbf x_{\textup{s}}  }}{\partial  \mathbf a_i } = \frac{\partial A_R(\mathbf a_i)^{-1}( \mathbf x_{\textup{s}}  )}{\partial  \mathbf a_i } \cdot \nabla \tilde u(A_R(\mathbf a_i)^{-1}( \mathbf x_{\textup{s}}  ) ). \label{eq:TransformOnCoordsDeriv}
\end{equation}
Evaluating these derivatives for all sampling points $\mathbf x_{\textup{s}}\in D_{\mbf{u}}$ then enables the computation of the sought derivatives $\frac{\partial W_{i}}{\partial \mathbf{a}_i}$ via \eqref{eq:ProjRotateAndSum}.

We emphasize that the tricubic interpolator has just the right regularity to match the lesser differentiability conditions presented in section \ref{SS:lessDiff}.
% , while it does not require the costly solution of large linear systems as would be the case for $\mathscr{C}^2$-smooth cubic B-spline interpolators.
According to  in \eqref{eq:NLinCubicInterp}, however, this interpolation method requires evaluations of the kernel $k_{\textup{3d}}$ for the 64 nearest- and next-to-nearest neighbors of the  evaluation point.

Faster evaluations of the rotation- and shift map $R(\mathbf a_i)(\mbf{u})$ can be achieved by decomposing it into three subsequent transforms along 2D-coordinate planes, i.e.\ we exploit that - up to interpolation errors - the following relation holds:
\begin{align}
 R(\mathbf a_i)(\mbf{u}) &= R(\mathbf{a}_{i, xy}) R(\mathbf{a}_{i, yz}) R(\mathbf{a}_{i, zx})(\mbf{u}) \label{eq:3Dto2DTransforms} \\
 \mathbf{a}_{i, xy} := &(s_{i,x}, s_{i,y}, 0, \theta_{i, xy}, 0, 0), \quad \mathbf{a}_{i, yz} = (0, 0, s_{i,z}, 0, \theta_{i, yz}, 0), \quad \mathbf{a}_{i, zx} = (0, 0, 0, 0, 0, \theta_{i, zx})  \nonumber
\end{align}
The transforms are chosen such that each of them requires interpolation only within a specific coordinate plane of the input 3D data. In the case of $R(\mathbf{a}_{i, xy})$ for instance, all interpolation points are of the form $\mathbf x = (x,y,z)$ with $z \in \{1, \ldots, N_z\}$. \eqref{eq:NLinCubicInterp} becomes
\begin{align}
 \tilde u(\mathbf x ) &= \sum_{\mathbf x_{\textup{s}}\in D_{\mbf{u}}} u_{\mathbf x_{\textup{s}}  } k (x - x_{\textup{s}}) k (y - y_{\textup{s}}) \underbrace{k (z - z_{\textup{s}})}_{= \delta_{z,z_{\textup{s}}}} = \sum_{x_{\textup{s}}, y_{\textup{s}} = 1}^{N_x, N_y} u_{ (x_{\textup{s}}, y_{\textup{s}}, z  ) } k (x - x_{\textup{s}}) k (y - y_{\textup{s}}) \label{eq:3Dto2DInterp}
\end{align}
i.e.\ the 3D-interpolation is reduced to 2D. Moreover, note that the interpolation weights $k (x - x_{\textup{s}}) k (y - y_{\textup{s}})$ are independent of $z$. This enables efficient numerical implementation of \eqref{eq:3Dto2DInterp} by assembling the weights into a sparse matrix  and then applying it to a suitably reshaped version of $\mbf u$ by a sparse-matrix-dense-matrix-product.
%$\left(k (x - x_{\textup{s}}) k (y - y_{\textup{s}}) \right)_{(x,y), (x_{\textup{s}}, y_{\textup{s}})}$

We refer to the implementation scheme given by \eqref{eq:3Dto2DTransforms} and \eqref{eq:3Dto2DInterp} as \emph{bilinear} (for $k = k_{\textup{lin}}$) or \emph{bicubic} (for $k = k_{\textup{cubic}}$) due to the two-dimensional nature of the involved interpolations. Note that these methods result in the same smoothness of $\mathbf{a}_i \mapsto W_{i}(\mathbf{a}_i)$ as trilinear and tricubic interpolation, respectively. However, interpolation errors are probably more pronounced for the former ones due to the concatenation of interpolators.

\subsection{Tomographic reconstruction method} \label{SS:tomoReconMethod}

In the numerical examples presented in section \ref{S:Results}, we make the simple choice of quadratic Tikhonov regularization with a gradient penalty term as the cost function in \eqref{eq:joint}:
\begin{equation}
  f(\mbf{a},\mbf{u}) = \frac 1 2 \norm{W(\mbf a )  \mbf{u} - \mbf p ^{\textup{obs}} }_2^2 + \frac \alpha 2 \norm{ \nabla_{\textup{dis}} \mbf{u} }_{2}^2 \label{eq:TikhonovFunctional}
\end{equation}
Here, %$\mbf p ^{\textup{obs}}$ denotes the measured projection data and
%$\norm{ \mbf{u} }_{H^1}^2 := (1-\gamma) \norm{ \mbf{u} }_{2}^2 + \gamma \norm{ \nabla_{\textup{dis}} \mbf{u} }_{2}^2$, where $\gamma \in [0; 1]$ and
$\nabla_{\textup{dis}}$ denotes the discrete gradient operator. With this choice, the first step in Algorithms \ref{algorithm:smooth} -- \ref{algorithm:alternating} amounts to solving the quadratic least-squares problem
\begin{equation}
  \mbf{u}^{(k+1)}  =   \argmin_{ \mbf u } \norm{W(\mbf a^{(k)} )  \mbf{u} - \mbf p ^{\textup{obs}} }_2^2 +  \alpha   \norm{ \nabla_{\textup{dis}} \mbf{u} }_{2}^2. \label{eq:TikhonovTomoRec}
\end{equation}
We approximate minimizers by applying the conjugate gradient method to the normal equation corresponding to \eqref{eq:TikhonovTomoRec}.

As has been extensively discussed in the literature on Tikhonov regularization (see e.g.\ \cite{EnglEtAl1996_RegInvProb} and references therein), the penalty term in \eqref{eq:TikhonovTomoRec} establishes robustness of $\mbf{u}^{(k+1)}$ against errors in the data $\mbf p ^{\textup{obs}}$ and thus in particular to inconsistencies due to misalignment. Accordingly, it prevents overfitting in the tomographic reconstruction step. Our motivation for penalizing the discrete gradient of $\mbf u$ instead of choosing the standard penalty term $\norm{\mbf u}_2 ^2$ is two-fold: for once, misalignment often manifests in sharp stripe artifacts in the reconstruction, which can be suppressed more efficiently by smoothing penalties. On the other hand, it can be seen from \eqref{eq:TransformOnCoordsDeriv} that the evaluation of derivatives $\frac{\partial W_{i}}{\partial \mathbf{a}_i}$ requires approximations of the gradient of $\mbf u$. Penalizing large values of $\nabla_{\textup{dis}} \mbf u$ in the tomographic reconstruction thus stabilizes subsequent alignment steps.
%, see \ref{SS:alignMethod}.
% Note that the property of the total varation penalty $\norm{ \nabla_{\textup{dis}} \mbf{u} }_1$ to produce sharp edges and thus large sparse gradients in the reconstruction, is not desirable in the considered setting due to the latter smoothness requirement.

Minimizing the Tikhonov functional in \eqref{eq:TikhonovTomoRec} can be quite costly if the system is ill-conditioned. Moreover, the approach does not allow to impose a non-negativity constraint for the reconstructed object as is often reasonable to suppress artifacts. We therefore propose an alternative Kaczmarz-type scheme for the tomographic reconstruction in Algorithms \ref{algorithm:smooth} -- \ref{algorithm:alternating}:
\begin{align}
  \mbf{u}^{(k)}_{0} &:= \mbf{u}^{(k)}, \quad
  \mbf{u}^{(k)}_{j+1} =   \argmin_{ \mbf u } \norm{W(\mbf a^{(k)}_{m_j} )  \mbf{u} - \mbf p ^{\textup{obs}}_{m_j} }_2^2  +  \frac \alpha 2 \norm{ \nabla_{\textup{dis}} (\mbf{u} - \mbf{u}^{(k)}_{j} ) }_{2}^2 \nonumber \\
  \mbf{u}^{(k+1)} &:= \mbf{u}^{(k)}_{N} \label{eq:KaczmarzTomoRec}
\end{align}
We always compute one symmetric Kaczmarz cycle, i.e.\ $N =2 N_{\textup{proj}}$ is twice the number of tomographic projections and $m_1 = m_N, m_2 = m_{N-1}, \ldots m_{N/2} = m_{N/2+1}$ holds true. The processing order of the projections $\mbf p ^{\textup{obs}}_{m_1}, \mbf p ^{\textup{obs}}_{m_2}, \ldots $ is chosen according to the multi-level-scheme proposed in \cite{0031-9155-39-11-013}. It is shown in \cite{KindermannLeitao2014_genARTConvergence} that the final Kaczmarz-iterate $\mbf{u}^{(k)}_{N}$ minimizes of a preconditioned form of the Tikhonov functional in \eqref{eq:TikhonovTomoRec}. Accordingly, the proposed scheme matches our optimization setting \eqref{eq:joint}. Non-negativity of $\mbf{u}^{(k+1)}$  can be imposed by setting negative values to zero after each Kaczmarz-iteration in \eqref{eq:KaczmarzTomoRec}.

For the choice of the regularization parameter $\alpha$ in \eqref{eq:TikhonovTomoRec} and \eqref{eq:KaczmarzTomoRec}, we propose a heuristic rule: Let $\delta a$ denote an estimate for the perturbations of the alignment parameters to be corrected, i.e.\ $\delta a  \approx \norm{\mathbf{a}^{\text{(exact)}} - \mathbf{a}^{(0)}}_\infty$. Then tomographic reconstruction from the misaligned data can be expected to result in artifacts roughly on the lenghtscale $\delta a$ and below. In order to suppress these, $\alpha$ is chosen such that the data fidelity- and regularization term in  \eqref{eq:TikhonovTomoRec} are balanced exactly at this scale, by ensuring $ \alpha \approx \norm{W(\mbf a )  \mbf{u}_{\delta a}   }_2^2 / \norm{ \nabla_{\textup{dis}} \mbf{u}_{\delta a} }_{2}^2$ for Fourier modes $\mbf u_{\delta a} = \exp(\textup i  \boldsymbol \xi \cdot \mathbf x )$ of frequency $|\boldsymbol \xi| = \pi / \delta a$. Using the Fourier-slice-theorem, a simple formula for the choice of $\alpha$ can be obtained from this requirement. Features of size $> \delta a$ are then resolved, whereas  misalignment artifacts on scales $< \delta a$ are damped out in the reconstructions. There is no guarantee that this will lead to an optimal choice for $\alpha$, especially when dealing with limited-angle or truncated data. However, the proposed heuristic has proven to be useful in practice, including the latter settings.

\subsection{Alignment method} \label{SS:alignMethod}

As shown in section \ref{S:DiscXray}, the derivative of the projector $W = (W_1, \ldots, W_{N_{\textup{proj}}})$ with respect to the alignment parameters $\mbf a = (\mbf a _1, \ldots, \mbf a_{N_{\textup{proj}}})$ exhibits a block-diagonal structure. Since we do not impose a joint regularizer for $\mbf a$, the alignment steps in Algorithms \ref{algorithm:smooth} -- \ref{algorithm:alternating} decouple into independent iterations for all of the $N_{\textup{proj}}$ projections:
\begin{equation}
 \mathbf{a}_i^{(k+1)} = \mathbf{a}_i^{(k)} - \gamma_i^{(k)} \nabla_{\mathbf{a}_i}f(\mathbf{a}^{(k)},\mathbf{u}^{(k+1)})  = \mathbf{a}_i^{(k)} - \gamma_i^{(k)} \mbf s_i^{(k)}
\end{equation}
% We choose the simplest possibly implementation of the quasi-Newton steps in Algorithms: setting $B_i^{(k)}$ as the identity we arrive at gradient descent steps
% \begin{equation}
%  \mathbf{a}_i^{(k+1)} = \mathbf{a}_i^{(k)} - \gamma_i^{(k)}  \nabla_{\mathbf{a}_i}f(\mathbf{a}^{(k)},\mathbf{u}^{(k+1)}) = \mathbf{a}_i^{(k)} - \gamma_i^{(k)} \mbf s_i^{(k)}
% \end{equation}
with $\mbf s_i^{(k)} := \transp{ \big( \nabla_{\mathbf{a}_i} W_i  (\mbf a_i^{(k)}) \mathbf{u}^{(k+1)} \big) } \cdot \big( W_i(\mbf a_i^{(k)}) \mathbf{u}^{(k+1)} -   \mbf p ^{\textup{obs}}_i \big)$.
The step size parameters $\gamma_i^{(k)}$ are chosen according to exact line search for quadratic functionals:
\begin{equation}
 \gamma_i^{(k)} = \frac{\norm{\mbf s_i^{(k)}}_{\textup{align}}^2}{\norm{\big(\nabla_{\mathbf{a}_i} W_i(\mbf a _i^{(k)}) \mathbf{u}^{(k+1)}\big) \cdot \mbf   s_i^{(k)}  }_2^2}.
\end{equation}
In order to avoid ill-conditioning due to different magnitudes of translational- and rotational alignment parameters, we measure the increment $s_i^{(k)}$ in a weighted 2-norm
\begin{equation}
 \norm{\mbf s_i^{(k)}}_{\textup{align}} = \norm{\mbf w \odot \mbf s_i^{(k)}}_{2}, \qquad w = (1,1,1, w_{xy}, w_{yz}, w_{zx}).
\end{equation}
Here, $\odot$ denotes pointwise multiplication and the weights $w_{xy}, w_{yz}, w_{zx} \in \mR_+$ for the rotation angles $(\theta_{i,xy}, \theta_{i,yz}, \theta_{i,zx})$ within the different planes, respectively, are chosen such that the rotation $\theta_{i,\ast}$ corresponds to a mean motion by $w_{xy} \cdot \theta_{i,\ast}$ object voxels.

As the cost function $f$ is not quadratic w.r.t.\ $\mbf a _i$, the exact linesearch method may fail to decrease $f$, i.e.\ one might have an increase of the residual
\begin{equation}
 \norm{W_i(\mbf a_i^{(k)} - \gamma_i^{(k)}\mbf s_i^{(k)} ) \mathbf{u}^{(k+1)} -   \mbf p ^{\textup{obs}}_i}_2^2  >  \norm{W_i(\mbf a_i^{(k)}  ) \mathbf{u}^{(k+1)} -   \mbf p ^{\textup{obs}}_i}_2^2
\end{equation}
In this case, we sequentially reduce the proposed stepsize $\gamma_i^{(k)}$ by a factors of two until the residual is decreased. Note however, that this residual reduction failure never occured for the test cases presented in this work.

\end{document}